\documentclass{amsart}
\usepackage{biblatex}
\usepackage{hyperref}
\usepackage{amssymb}
\usepackage{mathrsfs}
\usepackage{mathtools}
\usepackage{xifthen}
\usepackage{tikz,tikz-cd}

\hypersetup{
    colorlinks,
    linkcolor={red},
    citecolor={magenta},
    urlcolor={blue}
}

\makeatletter
\@namedef{subjclassname@2020}{\textup{2020} Mathematics Subject Classification}
\makeatother

\newtheorem{thm}{Theorem}

\newtheorem{prop}{Proposition}

\newtheorem{cor}{Corollary}

\newtheorem{lem}{Lemma}

\theoremstyle{definition}
\newtheorem{defn}{Definition}

\DeclareMathOperator{\pow}{\mathcal{P}} 
\providecommand{\N}{\mathbb{N}} 
\providecommand{\Z}{\mathbb{Z}} 
\DeclareMathOperator{\im}{Im} 
\DeclareMathOperator{\id}{id} 
\DeclareMathOperator{\orb}{Orb} 
\providecommand{\abs}[1]{\left\lvert#1\right\rvert} 
\providecommand{\norm}[1]{\left\lVert#1\right\rVert} 
\providecommand{\set}[2][]{
	\ifthenelse{\isempty{#1}}{
		\left\{#2\right\}
	}{
		\left\{\,#1\;\middle|\;#2\,\right\}}
	}

\providecommand{\Ca}{\mathscr{C}}
\providecommand{\Da}{\mathscr{D}}

\providecommand{\Ia}{\mathscr{I}}
\DeclareMathOperator{\ob}{Ob} 
\DeclareMathOperator{\mor}{Mor} 
\DeclareMathOperator{\morc}{\mathbf{Mor}} 
\DeclareMathOperator{\setcat}{\mathbf{Set}} 
\DeclareMathOperator{\subl}{\mathbf{Sub}} 
\DeclareMathOperator{\autg}{\mathbf{Aut}} 
\DeclareMathOperator{\funcl}{Fun} 
\DeclareMathOperator{\funcat}{\mathbf{Fun}} 
\providecommand{\op}{\mathrm{op}} 
\DeclareMathOperator{\yo}{Yo} 
\DeclareMathOperator{\finsetcl}{FinSet} 
\DeclareMathOperator{\finsetcat}{\mathbf{FinSet}} 

\DeclareMathOperator{\struct}{Struct} 
\DeclareMathOperator{\structc}{\mathbf{Struct}} 
\DeclareMathOperator{\paircl}{Pair} 
\DeclareMathOperator{\paircat}{\mathbf{Pair}} 
\DeclareMathOperator{\hypercl}{Hyp} 
\DeclareMathOperator{\hypercat}{\mathbf{Hyp}} 
\DeclareMathOperator{\cyo}{Ca} 
\DeclareMathOperator{\pola}{\mathbf{Pol}} 
\DeclareMathOperator{\pol}{Pol} 
\DeclareMathOperator{\sympol}{SymPol} 
\DeclareMathOperator{\sympola}{\mathbf{SymPol}} 
\DeclareMathOperator{\isostr}{IsoStr} 

\begin{document}
\title{Invariants of structures}
\author[C. Aten]{Charlotte Aten}
\address{Department of Mathematics\\
University of Denver\\Denver 80208\\USA}
\urladdr{\href{https://aten.cool}{https://aten.cool}}
\email{\href{mailto:charlotte.aten@du.edu}{charlotte.aten@du.edu}}
\thanks{This paper constitutes the second half of the author's PhD thesis. Thanks to Jonathan Pakianathan for all his advice and support over the years, without which this project wouldn't've been nearly as fruitful.}

\begin{abstract}
We give a categorification of the notion of a mathematical structure originally given by Bourbaki in their set theory textbook. We show that any isomorphism-invariant property of a finite structure can be computed by counting the number of isomorphic copies of small substructures it contains. Our main theorem in this direction is a generalization of the classical result of Hilbert about elementary symmetric polynomials generating the algebra of all symmetric polynomials. We also show that, for structures built from sets, the Yoneda functor extends to a canonical embedding of any such category of structures into an associated category of structures in the sense of classical model theory.
\end{abstract}

\maketitle

\tableofcontents

\section{Introduction}
When Bourbaki began writing the \emph{\'{E}l\'{e}ments de math\'{e}matique}, well before category theory had been introduced in algebraic topology, much less in the rest of mathematics, they sought to lay out in the first text of the series, \emph{Theory of Sets} a systematic description of mathematical structures as they would appear throughout the rest of the series\cite{bourbaki2004}. A simplified version of their treatment was that a \emph{structure} was a set, say \(A\), equipped with an indexed family \(\set{f_i}_{i\in I}\) of \emph{relations} \(f_i\) where each \(f_i\) was a subset of a set which could be constructed from \(A\) by taking Cartesian products and powersets finitely many times. Thus, denoting by \(\pow(A)\) the collection of subsets of \(A\), a relation on \(A\) might be a subset of
	\[
		A\times\pow(\pow(A)\times A^{57})\times\pow(\pow(\pow(A))),
	\]
for instance. Note that the now-usual relational structures of model theory are precisely these without allowing the powerset operator.

Bourbaki defined what we would now call morphisms of these structures and proved several results about them, all of which turned out to be of a categorical nature. This is only natural, since Eilenberg was a member of the group. Once his work with Mac Lane had established category theory, Grothendieck and then Cartier were asked to produce a category theory component for the \emph{\'{E}l\'{e}ments}, although if either did their contribution never made it into the texts. Discussions in ``La Tribu'' during the 1950s seem to indicate that Bourbaki felt much of the \emph{\'{E}l\'{e}ments} would have to be rewritten in order to accommodate the new notions from category theory. More damning for categories in the \emph{\'{E}l\'{e}ments} was the difficulty of synthesizing the structural and categorical viewpoints together. The consensus became that this task was not worth the effort\cite[p.328]{corry2004}.

While we don't comment on whether it would have been worth it for Bourbaki to include a fusion of the notions of category and structure in the \emph{\'{E}l\'{e}ments}, we do present one possible categorification of the concept of structure here. We postpone our formal development until \autoref{sec:structures_formally} and begin with a proof of a generalization of a result of Hilbert about symmetric polynomials\cite[p.191]{lang2005} to the setting of finite structures. This generalization is our \autoref{thm:elementary_symmetric_generate} and has the perhaps surprising implication (given as \autoref{cor:first_order}) that any first-order property of a finite structure \(\mathbf{A}\) can be checked by counting the number of embeddings of small substructures \(\mathbf{B}\hookrightarrow\mathbf{A}\), where ``small'' is a function of the logical complexity of the first-order property.

Once we have gone through our formal treatment of structures in \autoref{sec:structures_formally}, we can prove our final result in \autoref{sec:Yoneda}, which is \autoref{prop:cartesian_yo_full_faith} and which says that any category of structures built from sets may be embedded into a category of structures in the classical, model-theoretic sense by way of the Yoneda functor, providing one is willing to adopt a large signature.

This work is reminiscent of some similar topics where combinatorics and category theory overlap, although to this author's knowledge the ideas presented here haven't been addressed elsewhere. In particular, there is the theory of combinatorial species, where generating functions are associated to families of finite structures\cite{joyal1981}. While we will be counting finite structures here, it seems at the moment that there is only a superficial relationship with the results in that area.

In a related vein, the author would like to thank John Baez for pointing out the work of Lov\'{a}sz on structures, which concerns polynomials created by taking sums and products of finite structures themselves\cite{lovasz1967}.

The relationship between finite structures and certain classes of logical formulae has been treated elsewhere. See \cite{bova2013}, for instance, where methods from universal algebra are used to obtain some complexity results.

We make one final acknowledgment to Bill Lawvere (and universal algebra). The idea of categorifying Bourbaki's structures was very much inspired by Lawvere's thesis work on algebraic theories, a categorical treatment of universal algebra\cite{lawvere1963}. Any connection beyond the thematic is beyond the scope of the present work, but it would be interesting to examine the relationship between algebraic categories and categories of structures as defined here, particularly in light of the embedding result given as \autoref{prop:cartesian_yo_full_faith}.

This paper is organized thusly: Following this introduction, we treat structures in an intuitive way in \autoref{sec:structures_naively}. This will allow the reader to proceed on to most of the main results (\autoref{thm:elementary_symmetric_generate} and \autoref{cor:first_order}) without learning our general theory of structures. In \autoref{sec:symmetric_polynomials} we examine polynomials associated to finite structures and prove our aforementioned generalization of Hilbert's result on symmetric polynomials as \autoref{thm:elementary_symmetric_generate}. After this, we tie our results in to logic in \autoref{sec:logic} and prove our claim connecting these polynomials with combinatorics and logic as \autoref{cor:first_order}. Our penultimate \autoref{sec:structures_formally} gives our categorical treatment of Bourbaki's structures, and our final section closes with \autoref{prop:cartesian_yo_full_faith}, which says in part that (modulo accepting a large signature, and only for structures built from sets, and with all due credit to Yoneda) we could have just stuck with classical model theory avoided all of this.

\section{Structures, na\"{i}vely}
\label{sec:structures_naively}
As mentioned in the introduction, Bourbaki's original treatment of structures was fairly involved and presumably a categorification might be even more convoluted. Thus, we postpone this formal development until \autoref{sec:structures_formally} and present three possible ways for the reader to think about structures in this section.

The first is that one might consider a finite structure as one would in model theory. That is, a \emph{finite structure} is a pair  \(\mathbf{A}\coloneqq(A,\set{f_i}_{i\in I})\) where \(A\) is a finite set and the \(f_i\) form an \(I\)-indexed sequence of relations \(f_i\subset A^{\rho(i)}\) where the function \(\rho\colon I\to\N\) is the \emph{signature} of \(\mathbf{A}\). We denote by \(\structc^\rho\) the evident category and by \(\struct^\rho_A\) the collection of all structures of the same signature on the set \(A\), which we call a \emph{kinship class}. The class \(\struct^\rho\) of all structures with signature \(\rho\) is likewise called a \emph{similarity class}. We will always take the index set \(I\) to be finite here.

Reading the paper this way, one can take the statement \(N\in\ob(\Ia)\) to mean \(N\in I\). Similarly, \(F(N)\) indicates the basic relation \(f_N\). Wherever we write \(\rho_A(N)\) we mean nothing more that \(A^{\rho(N)}\). One may safely ignore statements about \(\mor(\Ia)\) and any corresponding discussion of \(F(\nu)\). While there may be a couple comments which remain mysterious in \autoref{sec:structures_naively} and \autoref{sec:symmetric_polynomials} given this perspective, it should be possible to glide over them without any great harm to understanding.

The second way of thinking about finite structures is in the sense of Bourbaki. That is, do the same as in the preceding paragraphs but allow yourself to think of a \emph{relation} as allowing powerset operators and Cartesian products in arbitrary finite compositions. This makes it easier to see how finite topological spaces could be counted among finite structures.

The third way, which should probably be postponed on first reading, is to instead go through \autoref{sec:structures_formally} first to see structures (not necessarily finite, or even built from sets at all) in their full, formal generality. This will make the following sections much more rigorous at the expense of adding extra bookkeeping.

\subsection{Substructures}
There is a natural categorical definition of a substructure. We will make use of this notion in order to disucss meets and joins in a corresponding lattice as part of the proof of \autoref{lem:factorization_lemma}.

\begin{defn}[Substructure]
Given a structure \(\mathbf{A}\) of signature \(\rho\) we refer to a subobject of \(\mathbf{A}\) in \(\structc^\rho\) as a \emph{substructure} of \(\mathbf{A}\).
\end{defn}

In the case of \(\setcat\)-structures we can give a concrete description of the poset of substructures. Given a \(\setcat\)-structure \(\mathbf{A}\coloneqq(A,F)\) of signature \(\rho\) we have that \(\mathbf{A}\) consists of, for each \(N\in\ob(\Ia)\), a subset \(F(N)\subset\rho_A(N)\), and, for each \(\nu\in\mor(\Ia)\) with domain \(N\), a restriction (on both the domain and codomain) \(F(\nu)=(\rho_A(\nu))|_{F(N)}\). This means that for a collection of subsets of the \(\rho_A(N)\) to form a structure of signature \(\rho\) it is necessary and sufficient that given a morphism \(\nu\colon N_1\to N_2\) from \(\Ia\) we have that the image of \(F(N_1)\) under \(\rho_A(\nu)\) is contained in \(F(N_2)\).

Note that a \(\setcat\)-structure with universe \(A\) is a substructure of the structure \((A,\id_{\rho_A})\). A substructure of some \(\mathbf{A}_1\coloneqq(A,F_1)\in\struct^\rho_A\) is then another structure \(\mathbf{A}_2\coloneqq(A,F_2)\in\struct^\rho_A\) such that \(\mathbf{A}_2\le\mathbf{A}_1\) in the substructure poset of \((A,\id_{\rho_A})\), which is equivalent to having for each \(N\in\ob(\Ia)\) that \(F_2(N)\subset F_1(N)\).

One can verify that the substructure poset of \((A,\id_{\rho_A})\) forms a complete lattice and hence the substructure poset \(\subl(\mathbf{A})\) of any \(\mathbf{A}\in\struct^\rho_A\) is also a complete lattice. Since the substructure poset of \((A,(\id_{\rho_A}))\) is a sublattice of a Boolean lattice all of these lattices are also distributive.

\subsection{Finite structures}
We give formal definitions of finite structures here, but if you're thinking of model-theoretic or Bourbakian structures you can ignore these definitions in favor of the ones you already have in mind.

\begin{defn}[Finite signature]
We say that a signature \(\rho\colon\Ia\to\funcat(\setcat,\setcat)\) is \emph{finite} when \(\Ia\) has finitely many objects and finitely many morphisms and for each \(N\in\ob(\Ia)\) and each finite set \(A\) we have that \(\rho_A(N)\) is finite.
\end{defn}

\begin{defn}[Finite structure]
We say that a structure of finite signature \(\rho\) on a finite set is a \emph{finite structure}.
\end{defn}

\begin{defn}[Finite kinship class]
When \(\rho\) is a finite signature and \(A\) is a finite set we say that \(\struct^\rho_A\) is a \emph{finite kinship class}.
\end{defn}

Note that each of the members of a finite kinship class are finite structures and that the kinship class itself is a finite set.

\section{Symmetric polynomials}
\label{sec:symmetric_polynomials}
We prove here, as \autoref{thm:elementary_symmetric_generate}, our generalization of Hilbert's classical result on symmetric polynomials. In order to do this, we consider polynomial algebras associated to finite kinship classes.

\subsection{Definitions}
As will be elucidated further in \autoref{sec:logic}, we will be building polynomials from variables which are meant to check whether a particular element of \(\rho_A(N)\) belongs to \(F(N)\) or not. For example, if we are considering the case where \(\rho_A(N)=\binom{A}{2}\) then we will introduce variables \(x_{N,\set{a_1,a_2}}\) where \(\set{a_1,a_2}\in\binom{A}{2}\).

\begin{defn}[Variables \(X^\rho_A\)]
Given a finite signature \(\rho\) on an index category \(\Ia\) and a finite set \(A\) we define
	\[
		X^\rho_A\coloneqq\bigcup_{\mathclap{N\in\ob(\Ia)}}\set[x_{N,a}]{a\in\rho_A(N)}.
	\]
\end{defn}

Given a commutative ring \(\mathbf{R}\) and a set \(X\) we write \(\mathbf{R}[X]\) to denote the free commutative unital \(\mathbf{R}\)-algebra generated by \(X\) and \(R[X]\) to denote that algebra's universe.

\begin{defn}[Monomial \(y_{\mathbf{A}}\)]
Given a finite signature \(\rho\) on an index category \(\Ia\), a finite set \(A\), and a structure \(\mathbf{A}\coloneqq(A,F)\in\struct^\rho_A\) we define \[y_{\mathbf{A}}\coloneqq\prod_{N\in\ob(\Ia)}\prod_{a\in F(N)}x_{N,a}.\]
\end{defn}

Note that there is always an empty structure of a given signature and hence one of the \(y_{\mathbf{A}}\) will always be \(1\).

\begin{defn}[Monomials \(Y^\rho_A\)]
Given a finite signature \(\rho\) on an index category \(\Ia\) and a finite set \(A\) we define \[Y^\rho_A\coloneqq\set[y_{\mathbf{A}}]{\mathbf{A}\in\struct^\rho_A}.\]
\end{defn}

\begin{defn}[\((\rho,A)\) polynomial algebra]
Given a commutative ring \(\mathbf{R}\), a finite signature \(\rho\), and a finite set \(A\) we define the \emph{\((\rho,A)\) polynomial algebra} over \(\mathbf{R}\) to be the subalgebra of \(\mathbf{R}[X^\rho_A]\) which is generated by \(Y^\rho_A\). We denote this algebra by \(\pola^\rho_A(\mathbf{R})\) and its universe by \(\pol^\rho_A(\mathbf{R})\).
\end{defn}

We omit the ring \(\mathbf{R}\) when we take \(\mathbf{R}\) to be \(\Z\). For example, we write \(\pola^\rho_A\) to indicate \(\pola^\rho_A(\Z)\).

By our previous comment that \(1\in Y^\rho_A\) polynomials in \(\pol^\rho_A(\mathbf{R})\) can have any nonzero constant term.

In order to prove the main result of this section we will need the following lemma on the factorization of monomials in \(\pola^\rho_A(\mathbf{R})\).

\begin{lem}
\label{lem:factorization_lemma}
Given \(y_{\mathbf{A}_1},\dots,y_{\mathbf{A}_k}\in Y^\rho_A\) we have that \[\prod_{i=1}^ky_{\mathbf{A}_i}=y_{\bigvee_{i=1}^k\mathbf{A}_i}\mu\] where \(\mu\in\pol^\rho_A\).
\end{lem}

\begin{proof}
We induct on the number of factors \(k\). When \(k=1\) we can take \(\mu=1\) and when \(k=2\) we can take \(\mu=y_{\mathbf{A}_1\wedge\mathbf{A}_2}\). Take \(k\ge3\) and suppose that we have the result for all \(k'<k\). In this case we observe that
	\begin{align*}
		\left(\prod_{i=1}^{k-1}y_{\mathbf{A}_i}\right)y_{\mathbf{A}_k} &= \left(y_{\bigvee_{i=1}^{k-1}\mathbf{A}_i}\mu\right)y_{\mathbf{A}_k} \\
		&= \left(y_{\bigvee_{i=1}^{k-1}\mathbf{A}_i}y_{\mathbf{A}_k}\right)\mu \\
		&= \left(y_{\left(\bigvee_{i=1}^{k-1}\mathbf{A}_i\right)\vee\mathbf{A}_k}\mu'\right)\mu \\
		&= y_{\bigvee_{i=1}^k\mathbf{A}_i}\mu\mu'.
	\end{align*}
Since \(\mu,\mu'\in\pol^\rho_A\) we have that \(\mu\mu'\in\pol^\rho_A\), as well.
\end{proof}

We have a natural action of \(\mathbf{\Sigma}_A\) on \(\mathbf{R}[X^\rho_A]\).

\begin{defn}[Action \(\upsilon\)]
We define a group action \(\upsilon\colon\mathbf{\Sigma}_A\to\autg(\mathbf{R}[X^\rho_A])\) by setting \((\upsilon(\sigma))(x_{N,a})\coloneqq x_{N,(\rho_\sigma(N))(a)}\) and extending.
\end{defn}

\begin{defn}[Symmetric polynomial]
A polynomial \(p\in\pol^\rho_A(\mathbf{R})\) is called \emph{symmetric} when for every \(\sigma\in\Sigma_A\) we have that \((\upsilon(\sigma))(p)=p\).
\end{defn}

\begin{defn}[\((\rho,A)\) symmetric polynomial algebra]
We denote by \(\sympol^\rho_A(\mathbf{R})\) the set of symmetric polynomials in \(\pol^\rho_A(\mathbf{R})\) and we denote by \(\sympola^\rho_A(\mathbf{R})\) the subalgebra of \(\pola^\rho_A(\mathbf{R})\) with universe \(\sympol^\rho_A(\mathbf{R})\), which we refer to as the \emph{\((\rho,A)\) symmetric polynomial algebra} over \(\mathbf{R}\).
\end{defn}

Of particular interest are a family of polynomials arising from the isomorphism classes of members of \(\struct^\rho_A\).

\begin{defn}[Action \(\zeta\)]
We define a group action \(\zeta\colon\mathbf{\Sigma}_A\to\mathbf{\Sigma}_{\struct^\rho_A}\) by \[(\zeta(\sigma))(A,F)\coloneqq(A,\rho_\sigma\circ F).\]
\end{defn}

This action is well-defined as a change in representative \(F\) doesn't change the equivalence class of monomorphisms to which \(\rho_\sigma\circ F\) belongs.

\begin{defn}[Isomorphism classes of structures]
We define \[\isostr^\rho_A\coloneqq\set[\orb_\zeta(\mathbf{A})]{\mathbf{A}\in\struct^\rho_A}.\]
\end{defn}

\begin{defn}[Elementary symmetric polynomial]
Given a finite signature \(\rho\), a finite set \(A\), and an isomorphism class \(\psi\in\isostr^\rho_A\) we define the \emph{elementary symmetric polynomial} of \(\psi\) to be \[s_\psi\coloneqq\sum_{\mathbf{A}\in\psi}y_{\mathbf{A}}.\]
\end{defn}

\begin{defn}[Polynomials \(S^\rho_A\)]
Given a finite signature \(\rho\) and a finite set \(A\) we define \[S^\rho_A\coloneqq\set[s_\psi]{\psi\in\isostr^\rho_A}.\]
\end{defn}

\begin{prop}
The elementary symmetric polynomials are symmetric polynomials.
\end{prop}

\begin{proof}
Let \(s_\psi\) be an elementary symmetric polynomial over \(\mathbf{R}\). Since \(s_\psi\) is a sum of monomials belonging to \(Y^\rho_A\) we have that \(s_\psi\in\pol^\rho_A(\mathbf{R})\). Take \(\sigma\in\Sigma_A\). We have that
	\begin{align*}
		(\upsilon(\sigma))(s_\psi) &= (\upsilon(\sigma))\left(\sum_{(A,F)\in\psi}\prod_{N\in\ob(\Ia)}\prod_{a\in F(N)}x_{N,a}\right) \\
		&= \sum_{(A,F)\in\psi}\prod_{N\in\ob(\Ia)}\prod_{a\in F(N)}(\upsilon(\sigma))(x_{N,a}) \\
		&= \sum_{(A,F)\in\psi}\prod_{N\in\ob(\Ia)}\prod_{a\in F(N)}x_{N,(\rho_\sigma(N))(a)} \\
		&= \sum_{\substack{(\zeta(\sigma))(A,F)\\(A,F)\in\psi}}\prod_{N\in\ob(\Ia)}\prod_{a\in(\rho_\sigma\circ F)(N)}x_{N,a} \\
		&= \sum_{(A,F)\in\psi}\prod_{N\in\ob(\Ia)}\prod_{a\in F(N)}x_{N,a} \\
		&= s_\psi,
	\end{align*}
as claimed.
\end{proof}

\subsection{A generating set}
We introduce a notion of weight for polynomials in our context and prove our main theorem on symmetric polynomials.

\begin{defn}[Magnitude of a structure]
Given a finite structure \(\mathbf{A}\coloneqq(A,F)\in\struct^\rho_A\) we define the \emph{magnitude} of \(\mathbf{A}\) to be \[\norm{\mathbf{A}}\coloneqq\sum_{N\in\ob(\Ia)}\abs{F(N)}.\]
\end{defn}

\begin{defn}[Magnitude of an isomorphism class]
Given \(\psi\in\isostr^\rho_A\) we define the \emph{magnitude} of \(\psi\) to be \(\norm{\psi}\coloneqq\norm{\mathbf{A}}\) for any \(\mathbf{A}\in\psi\).
\end{defn}

Since isomorphic structures have the same magnitude \(\norm{\psi}\) is well-defined.

\begin{prop}
We have that \(s_\psi\) is homogeneous of degree \(\norm{\psi}\).
\end{prop}

\begin{proof}
Observe that \(s_\psi\) is a sum of monomials, one for each member \(\mathbf{A}\) of \(\psi\). Each of these monomials have degree \(\norm{\mathbf{A}}=\norm{\psi}\).
\end{proof}

\begin{defn}[Variables \(Z^\rho_A\)]
Given a finite signature \(\rho\) on an index category \(\Ia\) and a finite set \(A\) we define \[Z^\rho_A\coloneqq\set[z_\psi]{\psi\in\isostr^\rho_A}.\]
\end{defn}

\begin{defn}[Weight of a monomial]
The \emph{weight} of a monomial \(\prod_\psi z_\psi^{d_\psi}\) in \(R[Z^\rho_A]\) is defined to be \(\sum_\psi\norm{\psi}d_\psi\).
\end{defn}

\begin{defn}[Weight of a polynomial]
The weight of a polynomial \(p\in R[Z^\rho_A]\) is the maximum of the weights of the monomials appearing in \(p\).
\end{defn}

We generalize a statement of Hilbert by showing that the elementary symmetric polynomials generate the algebra of symmetric polynomials. We follow Lang's treatment\cite[p.191]{lang2005}.

\begin{thm}
\label{thm:elementary_symmetric_generate}
Given a polynomial \(f\in\sympol^\rho_A(\mathbf{R})\) of degree \(d\) there exists a polynomial \(g\in R[Z^\rho_A]\) of weight at most \(d\) such that \(f=g|_{Z^\rho_A=S^\rho_A}\).
\end{thm}

\begin{proof}
Define \(n\coloneqq\abs{A}\). We induct on \(n\). When \(n=0\) we have that \(\mathbf{\Sigma}_A\) is trivial and hence each class in \(\isostr^\rho_A\) contains a unique member. It follows that \(\sympol^\rho_A(\mathbf{R})=\pol^\rho_A(\mathbf{R})\) and the \(Y^\rho_A\) are precisely the elementary symmetric polynomials. The polynomial \(g\) can be obtained from \(f\) by replacing each occurrence of a monomial \(y_{\mathbf{A}}\) in a term of \(f\) with the corresponding singleton orbit variable \(z_{\set{\mathbf{A}}}\). By definition of the weight of a polynomial this choice of \(g\) will have weight precisely \(d\).

Suppose that \(n>0\) and that we have the result for \(n-1\). We induct on \(d\). Put a total order on \(A\) so that \(A=\set{a_1,\dots,a_n}\). Define \(B\coloneqq A\setminus\set{a_n}\) and define \(\iota\colon B\to A\) to be inclusion given by \(\iota(a_i)\coloneqq a_i\). For each \(N\in\ob(\Ia)\) this map induces an inclusion \(\rho_\iota(N)\colon\rho_B(N)\to\rho_A(N)\). Define
	\[
		A_n\coloneqq\bigcup_{\mathclap{N\in\ob(\Ia)}}\set[x_{N,a}]{a\in\rho_A(N)\setminus\im(\rho_\iota(N))}
	\]
to be the collection of variables in \(X^\rho_A\) depending on \(a_n\). We have that \(f|_{A_n=0}\in\sympol^\rho_B(\mathbf{R})\) so there exists some \(g_1\in R[Z^\rho_B]\) of weight at most \(d\) such that \(f|_{A_n=0}=g_1|_{Z^\rho_B=S^\rho_B}\). Note that for each \(s_\psi\in S^\rho_B\) there is a unique member \(s'_\psi\in S^\rho_A\) such that \(s_\psi=(s'_\psi)|_{A_n=0}\). By the inclusion induced by \(\iota\) identify \(g_1\) with a polynomial, which we will also call \(g_1\), belonging to \(R[Z^\rho_A]\). We find that \(f|_{A_n=0}=(g_1|_{Z^\rho_A=S^\rho_A})|_{A_n=0}\). Define \(f_1\coloneqq f-g_1|_{Z^\rho_A=S^\rho_A}\). Observe that \(f_1\) has degree at most \(d\) and is symmetric.

By applying our lemma on the factorization of monomials we can write \(f_1\) uniquely as
	\[
		f_1=\sum_{\mathclap{\mathbf{A}\in\struct^\rho_A}}y_{\mathbf{A}}p_{\mathbf{A}}
	\]
where each \(y_{\mathbf{A}}\) is the monomial factor guaranteed by that lemma. Since \(f_1\) has no constant term each \(p_{\mathbf{A}}\in\pol^\rho_A(\mathbf{A})\) has degree strictly less than \(d\). Since \(f_1\) is symmetric the application of some \(\sigma\in\Sigma_A\) would appear to give us a different such expression for \(f_1\). It follows that if \(\mathbf{A}_1\cong\mathbf{A}_2\) then \(p_{\mathbf{A}_1}=p_{\mathbf{A}_2}\). We can then collect terms to obtain \(f_1=\sum_{\psi\in\isostr^\rho_A}s_\psi p_\psi\) where \(p_\psi=p_{\mathbf{A}}\) for any \(\mathbf{A}\in\psi\). Applying the inductive hypothesis to the \(p_{\mathbf{A}}\) we obtain \(p_{\mathbf{A}}=(g_{\mathbf{A}})|_{Z^\rho_A=S^\rho_A}\) where each \(g_{\mathbf{A}}\) has weight at most \(d-\norm{\psi}\). It follows that \(f_1\) can be written as a polynomial \(g_2\) in the elementary symmetric polynomials of weight at most \(d\). That is, there exists some \(g_2\in R[Z^\rho_A]\) of weight at most \(d\) such that \(f_1=g_2|_{Z^\rho_A=S^\rho_A}\). This implies that \(f=(g_1+g_2)|_{Z^\rho_A=S^\rho_A}\) where \(g_1+g_2\in R[Z^\rho_A]\) has weight at most \(d\).
\end{proof}

Although it happens that \((f_1)|_{A_n=0}=0\) and hence each term in \(f_1\) is divisible by some element of \(A_n\) we did not need to use this fact. We can make an observation analogous to the one in the proof in Lang (loc. cit.) in this direction, which is that by symmetry each term in \(f_1\) must be divisible by some element of \(A_i\) for each \(i\). This suffices in that special case because there is a minimal monomial with this property.

Most of our definitions and arguments go through if we use a suitably finite signature \(\rho\colon\Ia\to\funcat(\setcat,\setcat^\op)\) instead. If a similar result for this class of structures is to be proved then we must make a change at the point where we take the induced map \(\rho_\iota(N)\), for this will now give us a morphism in \(\setcat^\op\) whose corresponding map in \(\setcat\) is one taking members of \(\rho_A(N)\) to members of \(\rho_B(N)\).

It is not the case in general that the symmetric polynomials \(S^\rho_A\) generate
    \[
        \sympola^\rho_A(\mathbf{R})
    \]
freely. We give a specific example of a nontrivial algebraic relation between elementary symmetric polynomials in the next section.

\subsection{Example: domain digraphs}
\label{subsec:domain_digraph}
Most of the categories of structures with which we are already acquainted have no nontrivial relators. We consider structures with a relator which is not an identity or isomorphism in order to get a flavor of the general case.

\begin{defn}[Domain digraph]
A \emph{domain digraph} with universe \(A\) consists of some \(E\subset A^2\) and some \(W\subset A\) such that for each \((a_0,a_1)\in E\) we have that \(\pi(a_0,a_1)=a_0\in W\).
\end{defn}

We can visualize this as a digraph \(E\) on a set of vertices \(A\) where a subset \(W\subset A\) of \emph{domain vertices} are marked. Each edge in \(E\) must have its source vertex in \(W\), although in general a domain vertex need not be the source of any edge in \(E\). We will denote a domain digraph \(\mathbf{A}\) with universe \(A\), edge set \(E\), and domain vertex set \(W\) by \(\mathbf{A}\coloneqq(A,E,W)\).

Domain digraphs can be construed as structures in our formal sense where there are two basic relations and a morphism between them in the index category \(\Ia\).

We give an example where the elementary symmetric polynomials \(S^\rho_A\) are algebraically dependent. Take \(A\coloneqq\set{x_0,x_1}\). In this case we have that \(A^2=\set{x_{00},x_{01},x_{10},x_{11}}\). Observe that \[X^\rho_A=\set{x_0,x_1,x_{00},x_{01},x_{10},x_{11}}\] and
	\begin{align*}
		Y^\rho_A=\{ & 1,x_0,x_1,x_0x_1, \\
        & x_{00}x_0,x_{00}x_0x_1,x_{01}x_0,x_{01}x_0x_1,x_{10}x_1,x_{10}x_0x_1,x_{11}x_1,x_{11}x_0x_1, \\
        & x_{00}x_{01}x_0,x_{00}x_{01}x_0x_1,x_{00}x_{10}x_0x_1,x_{00}x_{11}x_0x_1, \\
        & x_{01}x_{10}x_0x_1,x_{01}x_{11}x_0x_1,x_{10}x_{11}x_1,x_{10}x_{11}x_0x_1, \\
        & x_{00}x_{01}x_{10}x_0x_1,x_{00}x_{01}x_{11}x_0x_1,x_{00}x_{10}x_{11}x_0x_1,x_{01}x_{10}x_{11}x_0x_1, \\
        & x_{00}x_{01}x_{10}x_{11}x_0x_1\}.
	\end{align*}
We find that
	\begin{align*}
		S^\rho_A = \{1,s_0,s_{0,1},s_{00,0},s_{00,0,1},s_{01,0},\dots\}.
	\end{align*}
One example of an algebraic dependence between the elementary symmetric polynomials is
	\begin{align*}
		s_{00,0}s_{01,0} &= (x_{00}x_0+x_{11}x_1)(x_{01}x_0+x_{10}x_1) \\
		&= x_{00}x_{01}x_0^2+x_{00}x_{10}x_0x_1+x_{01}x_{11}x_0x_1+x_{10}x_{11}x_1^2 \\
		&= (x_{00}x_{01}x_0+x_{10}x_{11}x_1)(x_0+x_1)-(x_{00}x_{01}x_0x_1+x_{10}x_{11}x_0x_1)+ \\
        &{\hspace{1cm}} (x_{00}x_{10}x_0x_1+x_{01}x_{11}x_0x_1) \\
		&= s_{00,01,0}s_0-s_{00,01,0,1}+s_{00,10,0,1}.
	\end{align*}
More succinctly, we have \[s_{00,0}s_{01,0}-s_{00,01,0}s_0+s_{00,01,0,1}-s_{00,10,0,1}=0.\] This situation is different from that of the classical elementary symmetric polynomials, which are algebraically independent.

\section{Logic by counting}
\label{sec:logic}
We describe how to use \autoref{thm:elementary_symmetric_generate} to check whether a first-order property holds for a given finite structure. In particular, we will see that we can verify whether any such first-order property holds for a finite structure by merely counting how many isomorphic copies of small substructures it contains and then evaluating a polynomial with integer coefficients at these values. It will turn out that ``small'' can be taken to mean either ``expressible by a low-complexity formula'' or ``verifiable locally'', and that these meanings always coincide.

In order to do this, we describe how to translate formulas in first-order logic into polynomials.

\begin{defn}[Kindred language]
Given a finite kinship class \(\struct^\rho_A\) the \emph{kindred language} \(\mathcal{L}^\rho_A\) is the negation-free fragment of the first-order language whose atomic formulae are those of the form \(\gamma(N,a)\) where \(N\in\ob(\Ia)\) and \(a\in\rho_A(N)\).
\end{defn}

That is, formulae from \(\mathcal{L}^\rho_A\) consist of those atoms \(\gamma(N,a)\) along with those formulae which can be formed by conjunction, disjunction, and quantification (either universal or existential) over the arguments of known such formulae.

We interpret the formula \(\gamma(N,a)\) as indicating the statement \(a\in F(N)\) for some structure \(\mathbf{A}\coloneqq(A,F)\) from the finite kinship class \(\struct^\rho_A\). That is, we say that
    \[
        \mathbf{A}\coloneqq(A,F)\models\Gamma\in\mathcal{L}^\rho_A
    \]
when the sentence \(\Gamma^{\mathbf{A}}\) obtained by replacing each instance of \(\gamma(N,a)\) with \(a\in F(N)\) in \(\Gamma\) holds. Note that this only makes sense for a choice of \(F\) such that \(F(N)\subset\rho_A(N)\), but this is the natural choice for \(F\) in the context of finite structures, in any case.

Note that our abnegation of negation in the language \(\mathcal{L}^\rho_A\) does not reduce our expressiveness from general first-order logic in the sense that given any property \(P\) of structures in \(\struct^\rho_A\), we can express ``\(\mathbf{A}\) is among the structures in \(\struct^\rho_A\) which satisfy property \(P\)'' in \(\mathcal{L}^\rho_A\) as
    \[
        \bigvee_{\mathclap{(A,F)\in P'}}\forall_{N\in\ob(\Ia)}\forall_{a\in F(N)}\gamma(N,a)
    \]
where \(P'\) is the set of unique representatives of the structures from \(\struct^\rho_A\) which have property \(P\) obtained by taking \(F(N)\) to be a subset of \(\rho_A(N)\) in each instance.

On a similar note, we lose nothing by dropping the universal and existential quantifiers, as well, since we are only quantifying over finite sets and we can always replace the quantifier with an appropriate conjunction or disjunction. Nevertheless, we will retain the quantifiers for the sake of readability.

We now give the map taking sentences in \(\mathcal{L}^\rho_A\) to \(\Z[X^\rho_A]\).

\begin{defn}[Polynomial realization]
The \emph{polynomial realization} function
    \[
        \xi\colon\mathcal{L}^\rho_A\to\Z[X^\rho_A]
    \]
is given by
    \[
        \xi(\Gamma)\coloneqq
            \begin{cases}
                x_{N,a} &\text{when } \Gamma=\gamma(N,a) \\
                \xi(\Gamma_1)+\xi(\Gamma_2) &\text{when } \Gamma=\Gamma_1\vee\Gamma_2 \\
                \xi(\Gamma_1)\xi(\Gamma_2) &\text{when } \Gamma=\Gamma_1\wedge\Gamma_2 \\
                \sum_{a\in\rho_A(N)}\xi(\Gamma'(N,a)) &\text{when } \Gamma=\exists_{a\in\rho_A(N)}\Gamma'(N,a) \\
                \prod_{a\in\rho_A(N)}\xi(\Gamma'(N,a)) &\text{when } \Gamma=\forall_{a\in\rho_A(N)}\Gamma'(N,a).
            \end{cases}
    \]
\end{defn}

Note that in general we don't have that \(\xi(\Gamma)\in\pol^\rho_A\). Consider the example of domain digraphs from \autoref{subsec:domain_digraph}. Given the universe \(A\coloneqq\set{x_0,x_1}\) from that example, we can have that
    \[
        \Gamma^{\mathbf{A}}=(x_{10}\in W)
    \]
and hence
    \[
        \xi(\Gamma)=x_{10},
    \]
but this monomial does not belong to \(\pol^\rho_A\) since it cannot be realized by a structure from \(\struct^\rho_A\). (Recall that, for domain digraphs, any monomial containing \(x_{10}\) must also have \(x_1\) as a factor.)

We adopt some more language in light of this map. We say that \(\Gamma\in\mathcal{L}^\rho_A\) is \emph{structural} when \(\xi(\Gamma)\in\pol^\rho_A\). Our preceding example of a formula \(\Gamma\) for the kindred language of domain digraphs with universe \(\set{x_0,x_1}\) was not structural. Such formulae are among those which can never be modeled by any of the structures in the relevant kinship class.

Finally, we will say that \(\Gamma\in\mathcal{L}^\rho_A\) is \emph{isomorphism-invariant} when \(\xi(\Gamma)\in\sympol^\rho_A\). Clearly, a formula may be structural without being isomorphism invariant. For instance, the property that a group with universe \(\set{a,b}\) has \(a\) as its identity element can be expressed by a structural formula, but not an isomorphism-invariant one.

We will say that a property \(P\) of structures from \(\struct^\rho_A\) is \emph{isomorphism-invariant} when there exists some \(\Gamma\in\mathcal{L}^\rho_A\) which is isomorphism-invariant in the above sense and for which \(\mathbf{A}\) satisfies \(P\) exactly when \(\mathbf{A}\models\Gamma\). Note that by our previous observations, this coincides with our informal notion of an isomorphism invariant property (i.e. isomorphism-invariant subset of \(\struct^\rho_A\)).

There is an evident evaluation map which takes polynomials from \(\mathbf{R}[X^\rho_A]\) to elements of \(R\).

\begin{defn}
Given any polynomial \(p\in\mathbf{R}[X^\rho_A]\) and a structure \(\mathbf{A}\in\struct^\rho_A\) we define \(p(\mathbf{A})\) by extension of the rule
    \[
        x_{N,a}(\mathbf{A})\coloneqq
            \begin{cases}
                1 &\text{when } \mathbf{A}\models\gamma(N,a) \\
                0 &\text{otherwise}
            \end{cases}
            .
    \]
\end{defn}

Before we reach the punchline of this section, we need one lemma.

\begin{lem}
\label{lem:positive_realization}
Given \(\Gamma\in\mathcal{L}^\rho_A\) and a structure \(\mathbf{A}\in\struct^\rho_A\), we have that \(\mathbf{A}\models\Gamma\) if and only if \((\xi(\Gamma))(\mathbf{A})>0\).
\end{lem}

\begin{proof}
We induct on the structure of \(\Gamma\). The base case is that \(\Gamma=\gamma(N,a)\). Since \(\xi(\Gamma)=x_{N,a}\), we have that \((\xi(\Gamma))(\mathbf{A})>0\) exactly when \(\mathbf{A}\models\gamma(N,a)=\Gamma\). Observe that for any structure \(\mathbf{A}\) and any formula \(\Gamma\) we have that \((\xi(\Gamma))(\mathbf{A})\ge0\).

When \(\Gamma=\Gamma_1\vee\Gamma_2\) we have that
    \[
        \xi(\Gamma)=\xi(\Gamma_1)+\xi(\Gamma_2).
    \]
By our inductive hypothesis, \(\mathbf{A}\models\Gamma_1\) if and only if \((\xi(\Gamma_1))(\mathbf{A})>0\) and \(\mathbf{A}\models\Gamma_2\) if and only if \((\xi(\Gamma_2))(\mathbf{A})>0\). It follows that \(\mathbf{A}\models\Gamma\) when either \((\xi(\Gamma_1))(\mathbf{A})>0\) or \((\xi(\Gamma_2))(\mathbf{A})>0\). Since
    \[
        (\xi(\Gamma))(\mathbf{A})=(\xi(\Gamma_1))(\mathbf{A})+(\xi(\Gamma_2))(\mathbf{A})
    \]
we find that \(\mathbf{A}\models\Gamma\) if and only if \((\xi(\Gamma))(\mathbf{A})>0\).

When \(\Gamma=\Gamma_1\wedge\Gamma_2\) we have that
    \[
        \xi(\Gamma)=\xi(\Gamma_1)\xi(\Gamma_2).
    \]
By our inductive hypothesis, \(\mathbf{A}\models\Gamma_1\) if and only if \((\xi(\Gamma_1))(\mathbf{A})>0\) and \(\mathbf{A}\models\Gamma_2\) if and only if \((\xi(\Gamma_2))(\mathbf{A})>0\). It follows that \(\mathbf{A}\models\Gamma\) when both \((\xi(\Gamma_1))(\mathbf{A})>0\) and \((\xi(\Gamma_2))(\mathbf{A})>0\). Since
    \[
        (\xi(\Gamma))(\mathbf{A})=(\xi(\Gamma_1))(\mathbf{A})(\xi(\Gamma_2))(\mathbf{A})
    \]
we find that \(\mathbf{A}\models\Gamma\) if and only if \((\xi(\Gamma))(\mathbf{A})>0\).

The arguments for the cases where
    \[
        \Gamma=\exists_{a\in\rho_A(N)}\Gamma'(N,a)
    \]
and
    \[
        \Gamma=\forall_{a\in\rho_A(N)}\Gamma'(N,a)
    \]
are essentially identical to those for disjunctions and conjunctions.
\end{proof}

Now we can state our corollary of \autoref{thm:elementary_symmetric_generate}.

\begin{cor}
\label{cor:first_order}
Suppose that \(P\) is an isomorphism-invariant property of structures from \(\struct^\rho_A\) such that a structure satisfies \(P\) exactly when that structure models the isomorphism-invariant formula \(\Gamma\in\mathcal{L}^\rho_A\). Such a property can be seen to hold (or not) for a particular \(\mathbf{A}\in\struct^\rho_A\) by counting substructures of \(\mathbf{A}\) and evaluating a polynomial with integer coefficients at the resulting values. Moreover, the size of the substructures needed is bounded by the complexity of the formula \(\Gamma\).
\end{cor}

\begin{proof}
By \autoref{lem:positive_realization} we have that \(\mathbf{A}\models\Gamma\) if and only if \((\xi(\Gamma))(\mathbf{A})>0\). We will show that \((\xi(\Gamma))(\mathbf{A})\) may be computed by counting substructures of \(\mathbf{A}\). By assumption, \(\xi(\Gamma)\in\sympol^\rho_A\) and therefore by \autoref{thm:elementary_symmetric_generate} we have that
    \[
        \xi(\Gamma)=g|_{Z^\rho_A=S^\rho_A}
    \]
for some polynomial \(g\in\Z[Z^\rho_A]\) of weight at most \(d\), where \(d\) is the degree of \(\xi(\Gamma)\). The polynomial \(g\) is the claimed polynomial with integer coefficients. It remains to examine the arguments of \(g\).

Observe that for each \(s_\psi\in S^\rho_A\) we have that
    \[
        s_\psi(\mathbf{A})=\abs{\set[\mathbf{B}\in\psi]{\mathbf{B}\le\mathbf{A}}}.
    \]
Thus, we have that \((\xi(\Gamma))(\mathbf{A})\) can be computed by evaluating \(g\) at arguments which are the counts of substructures of \(\mathbf{A}\) belonging to various isomorphism classes.

We proceed to address our claim on a bound for the sizes of the substructures needed for this calculation. More explicitly, we can bound \(\norm{\psi}\) such that \(z_\psi\) appears to a positive power in a monomial of \(g\) with a nonzero coefficient. Suppose that \(\Gamma\) can be written in disjunctive normal form as
    \[
        \Gamma=\bigvee_{i=1}^m\bigwedge_{j=1}^{n_m}\gamma(N_{ij},a_{ij})
    \]
and that no atomic formula \(\gamma(N_{ij},a_{ij})\) is repeated in any of the conjunctions therein. (Were this the case, we could obtain an equivalent formula by removing those redundant terms, and this situation makes our bound easier to state.) It follows that \(g\) has weight at most
    \[
        d\le\max_{1\le i\le m}n_m,
    \]
and therefore we only need to consider substructures belonging to isomorphism classes \(\psi\) with
    \[
        \norm{\psi}\le\max_{1\le i\le m}n_m
    \]
when performing our counts \(s_\psi(\mathbf{A})\) for computing \((\xi(\Gamma))(\mathbf{A})\).
\end{proof}

Note that the preceding result is only interesting because it is possible to obtain a nontrivial bound on \(\norm{\psi}\) such that \(z_\psi\) appears in the polynomial \(g\). Consider the polynomial realization of
    \[
        \bigvee_{\mathclap{(A,F)\in P'}}\forall_{N\in\ob(\Ia)}\forall_{a\in F(N)}\gamma(N,a)
    \]
from our cheap example of a formula expressing ``\(\mathbf{A}\) is among the structures in \(\struct^\rho_A\) which satisfy property \(P\)'' given earlier. Applying the method of the preceding proof to that polynomial, we discover that we are merely counting (if one can call it that) whether \(\mathbf{A}\) is equal to any structure which satisfies property \(P\). Clearly this can be done, but it's not terribly illuminating.

One can verify that for some properties, like whether a graph \(\mathbf{A}\) contains a \(3\)-cycle, it is possible to choose a corresponding \(\Gamma\) for which \(\Gamma^\mathbf{A}\) is true when \(\mathbf{A}\) contains a \(3\)-cycle without using the above trick. Indeed, one can see that only subgraphs of order at most \(3\) need to be examined in order to check whether \(\mathbf{A}\) has a \(3\)-cycle, and that this will correspond to the availability of a low-complexity choice of \(\Gamma\) (in that \(\max_{1\le i\le m}n_m\) is only \(3\)).

Our comment about locality should also be clear from this example. Checking for embedded copies of a subgraph of size \(3\) is a local check, since we have a bound on the amount of the graph we have to examine at once. On the other hand, checking whether a graph is isomorphic to a given graph is significantly nonlocal, in the sense that examining large subgraphs may not yield a direct affirmation or denial of the property in question. We would expect a corresponding formula \(\Gamma\) to always have a high complexity in such cases.

\section{Structures, formally}
\label{sec:structures_formally}
Here we provide the formal definitions for the notion of a mathematical structure used throughout this paper.

\subsection{Definition of a structure}
Given categories \(\Ca\) and \(\Da\) we denote by
    \[
        \funcl(\Ca,\Da)
    \]
the class of functors from \(\Ca\) to \(\Da\) and we denote by
    \[
        \funcat(\Ca,\Da)
    \]
the functor category from \(\Ca\) to \(\Da\).

\begin{defn}[Presignature]
Given an index category \(\Ia\) and categories \(\Ca\) and \(\Da\) we refer to a functor \(\rho\colon\Ia\to\funcat(\Ca,\Da)\) as a \emph{presignature}.
\end{defn}

To each presignature we associate another functor. Given a category \(\Ca\) we write \(\ob(\Ca)\) to indicate the class of objects of \(\Ca\) and \(\mor(\Ca)\) to indicate the class of morphisms of \(\Ca\).

\begin{defn}[Extractor]
Given a presignature \(\rho\colon\Ia\to\funcat(\Ca,\Da)\) the \emph{extractor} \(\rho_{\_}\colon\Ca\to\funcat(\Ia,\Da)\) is defined as follows. For \(A\in\ob(\Ca)\) we define \(\rho_A\colon\Ia\to\Da\) by \(\rho_A(N)\coloneqq(\rho(N))(A)\) for each \(N\in\ob(\Ia)\) and \(\rho_A(\nu)\coloneqq(\rho(\nu))_A\) for each \(\nu\in\mor(\Ia)\). For each morphism \(h\colon A\to B\) in \(\Ca\) we define \(\rho_h\colon\rho_A\to\rho_B\) by \((\rho_h)_N\coloneqq(\rho(N))(h)\) for each \(N\in\ob(\Ia)\).
\end{defn}

\begin{prop}
\label{prop:extractor_is_functor}
The extractor \(\rho_{\_}\colon\Ca\to\funcat(\Ia,\Da)\) of a presignature \(\rho\colon\Ia\to\funcat(\Ca,\Da)\) is a functor.
\end{prop}

\begin{proof}
We show that \(\rho_{\_}\) takes objects to objects. Given \(A\in\ob(\Ca)\) we show that \(\rho_A\colon\Ia\to\Da\) is a functor. Given \(N\in\ob(\Ia)\) we have that \(\rho(N)\colon\Ca\to\Da\) is a functor and hence \(\rho_A(N)=(\rho(N))(A)\in\ob(\Da)\). Given \(\nu\in\mor(\Ia)\) we have that \(\rho(\nu)\) is a natural transformation and hence \(\rho_A(\nu)=(\rho(\nu))_A\in\mor(\Da)\). Thus, \(\rho_A\) takes objects to objects and morphisms to morphisms. Observe that \[\rho_A(\id_N)=(\rho(\id_N))_A=(\id_{\rho(N)})_A=\id_{(\rho(N))(A)}=\id_{\rho_A(N)}\] so \(\rho_A\) takes identities to identities. Given morphisms \(\nu_1\colon N_1\to N_2\) and \(\nu_2\colon N_2\to N_3\) in \(\Ia\) we have that \[\rho_A(\nu_2\circ\nu_1)=(\rho(\nu_2\circ\nu_1))_A=(\rho(\nu_2)\circ\rho(\nu_1))_A=(\rho(\nu_2))_A\circ(\rho(\nu_1))_A=\rho_A(\nu_2)\circ\rho_A(\nu_1)\] so \(\rho_A\) respects composition of morphisms. Thus, \(\rho_A\) is a functor and \(\rho_{\_}\) takes objects to objects.

We show that \(\rho_{\_}\) takes morphisms to morphisms. Given a morphism \(h\colon A\to B\) in \(\Ca\) we show that \(\rho_h\colon\rho_A\to\rho_B\) is a natural transformation. Given a morphism \(\nu\colon N_1\to N_2\) in \(\Ia\) we have that \(\rho(\nu)\colon\rho(N_1)\to\rho(N_2)\) is a natural transformation so
	\begin{align*}
		\rho_B(\nu)\circ(\rho_h)_{N_1} &= (\rho(\nu))_B\circ(\rho(N_1))(h) \\
		&= (\rho(N_2))(h)\circ(\rho(\nu))_A \\
		&= (\rho_h)_{N_2}\circ\rho_A(\nu)
	\end{align*}
and hence \(\rho_h\) is also a natural transformation, which is a morphism in \(\funcat(\Ia,\Da)\).

We show that \(\rho_{\_}\) takes identities to identities. Given an object \(A\in\ob(\Ca)\) and an object \(N\in\ob(\Ia)\) we have that \[(\rho_{\id_A})_N=(\rho(N))(\id_A)=\id_{(\rho(N))(A)}=\id_{\rho_A(N)}\] so \(\rho_{\id_A}\) is the identity natural transformation of \(\rho_A\).

We show that \(\rho_{\_}\) respects composition of morphisms. Given morphisms \(h_1\colon A_1\to A_2\) and \(h_2\colon A_2\to A_3\) in \(\Ca\) and \(N\in\ob(\Ia)\) we have that \[(\rho_{h_2\circ h_1})_N=(\rho(N))(h_2\circ h_1)=(\rho(N))(h_2)\circ(\rho(N))(h_1)=(\rho_{h_2})_N\circ(\rho_{h_1})_N\] so \(\rho_{h_2\circ h_1}=\rho_{h_2}\circ\rho_{h_1}\), as desired.
\end{proof}

Recall the categorical formulation of images.

\begin{defn}[Factorization]
Given a morphism \(h\colon A\to B\) in a category \(\Ca\) we refer to a triple \((V,\theta,\psi)\) where \(V\in\ob(\Ca)\), \(\theta\colon A\to V\), \(\psi\colon V\to B\), and \(h=\psi\circ\theta\) as a \emph{factorization} of \(h\).
\end{defn}

\begin{defn}[Image candidate]
Given a morphism \(h\colon A\to B\) in a category \(\Ca\) we refer to a factorization \((V,\theta,\psi)\) of \(h\) as an \emph{image candidate} for \(h\) when \(\psi\) is monic.
\end{defn}

\begin{defn}[Image triple]
Given a morphism \(h\colon A\to B\) in a category \(\Ca\) we say that an image candidate \((V_1,\theta_1,\psi_1)\) is an \emph{image triple} for \(h\) when given any image candidate \((V_2,\theta_2,\psi_2)\) for \(h\) there exists a unique morphism \(s\colon V_1\to V_2\) such that \(\psi_2\circ s=\psi_1\).
\end{defn}

\begin{defn}[Image of a morphism]
Given a morphism \(h\colon A\to B\) in a category \(\Ca\) for which an image triple \((V,\theta,\psi)\) exists the \emph{image} \(\im(h)\) of \(h\) is the subobject of \(B\) containing \(\psi\).
\end{defn}

The image of a morphism is well-defined when it exists by the universal property of image triples.

We are interested in those presignatures which support taking images in a certain sense.

\begin{defn}[Signature]
Given a presignature \(\rho\colon\Ia\to\funcat(\Ca,\Da)\) we say that \(\rho\) is a \emph{\((\Ca,\Da)\)-signature} on the index category \(\Ia\) when given any monomorphism \(F\colon U\hookrightarrow\rho_A\) in \(\funcat(\Ia,\Da)\) and any morphism \(h\colon A\to B\) in \(\Ca\) we have that \(\im(\rho_h\circ F)\) exists in \(\funcat(\Ia,\Da)\). When \(\Ca=\Da\) we refer to a \((\Ca,\Da)\)-signature on \(\Ia\) as a \emph{\(\Ca\)-signature} on \(\Ia\).
\end{defn}

\begin{defn}[Source of a signature]
Given a signature \(\rho\colon\Ia\to\funcat(\Ca,\Da)\) we refer to \(\Ca\) as the \emph{source} of \(\rho\) and say that \(\rho\) is a \emph{\(\Ca\)-sourced} signature.
\end{defn}

\begin{defn}[Target of a signature]
Given a signature \(\rho\colon\Ia\to\funcat(\Ca,\Da)\) we refer to \(\Da\) as the \emph{target} of \(\rho\) and say that \(\rho\) is a \emph{\(\Da\)-targeted} signature.
\end{defn}

We give some examples of signatures. We denote by \(\Ia_I\) the category whose objects form the set \(\set[N_i]{i\in I}\) and whose morphisms are all identities. Given \(n\in\N\) we define \(\Ia_n\coloneqq\Ia_{\set{1,\dots,n}}\). We write \(N\) rather than \(N_1\) for the single object of \(\Ia_1\).

We make use of the following characterization of \(\funcat(\Ia_I,\Da)\) for any category \(\Da\).

\begin{defn}[Sequence category]
Given a set \(I\) and a category \(\Da\) the \emph{sequence category} \(\Da^I\) of \(\Da\) indexed by \(I\) is defined as follows. The objects of \(\Da^I\) are the \(I\)-indexed sequences \(\set{A_i}_{i\in I}\) of objects of \(\Da\). A morphism from \(\set{A_i}_{i\in I}\) to \(\set{B_i}_{i\in I}\) is an \(I\)-indexed sequence \(\set{h_i\colon A_i\to B_i}_{i\in I}\) of morphisms of \(\Da\). The identity morphism of \(\set{A_i}_{i\in I}\) is \(\set{\id_{A_i}\colon A_i\to A_i}_{i\in I}\). Composition of morphisms is performed componentwise. That is, if \(h_1\colon A_1\to A_2\) and \(h_2\colon A_2\to A_3\) are morphisms in \(\Da^I\) then we define \(h_2\circ h_1\colon A_1\to A_3\) by \((h_2\circ h_1)_i\coloneqq(h_2)_i\circ(h_1)_i\).
\end{defn}

In other words, \(\Da^I\) is the \(I^{\text{th}}\) direct power of the category \(\Da\).

It is evident that \(\funcat(\Ia_I,\Da)\) is canonically isomorphic to \(\Da^I\). Given \(n\in\N\) we define \(\Da^n\coloneqq\Da^{\set{1,\dots,n}}\). There is also a canonical isomorphism between \(\Da^1\) and \(\Da\) itself. Throughout we suppress these isomorphisms and speak of objects and morphisms of \(\Da^I\) rather than the corresponding functors from \(\Ia_I\) to \(\Da\) and their natural transformations wherever they appear.

Observe that any construction in \(\Da^I\) is a sequence of constructions in \(\Da\). Monomorphisms in \(\Da^I\) are sequences of monomorphisms in \(\Da\), factorizations in \(\Da^I\) are sequences of factorizations in \(\Da\), subobjects in \(\Da^I\) are sequences of subobjects in \(\Da\), and so forth.

We have a convenient criterion for a presignature to be a signature on \(\Ia_I\).

\begin{prop}
Suppose that \(\rho\colon\Ia_I\to\funcat(\Ca,\Da)\) is a presignature such that \(\Da\) has all images. We have that \(\rho\) is a signature.
\end{prop}

\begin{proof}
Suppose that \(F\colon U\hookrightarrow\rho_A\) is a monomorphism in \(\funcat(\Ia_I,\Da)\) and that \(h\colon A\to B\) is a morphism in \(\Ca\). Since \(\Da\) has all images each of the components \((\rho_h\circ F)_i\) of \(\rho_h\circ F\) has an image in \(\Da\) and this sequence of images is the image of \(\rho_h\circ F\) in \(\funcat(\Ia_I,\Da)\).
\end{proof}

All of the following signatures have index category \(\Ia_I\) for some \(I\). We will see signatures with more involved index categories later. Note that \(\setcat\) and \(\setcat^\op\) have all images.

\begin{defn}[Identity signature]
Given a category \(\Ca\) which has all images the \emph{identity signature} on \(\Ca\) is the functor \(\rho\colon\Ia_1\to\funcat(\Ca,\Ca)\) where \(\rho(N)\coloneqq\id_\Ca\) where \(\id_\Ca\) is the identity functor of \(\Ca\).
\end{defn}

\begin{defn}[\(n\)-set functor]
Given \(n\in\N\) denote by \(\binom{\_}{\le n}\) the functor from \(\setcat\) to \(\setcat\) which takes a set \(A\) to the collection \(\binom{A}{\le n}\coloneqq\bigcup_{i=1}^n\binom{A}{n}\) of nonempty subsets of size at most \(n\) in \(A\) and takes a function \(h\colon A\to B\) to the induced map from \(\binom{A}{\le n}\) to \(\binom{B}{\le n}\). We refer to \(\binom{\_}{\le n}\) as the \emph{\(n\)-set functor}.
\end{defn}

\begin{defn}[\(n\)-hypergraph signature]
The \emph{\(n\)-hypergraph signature} is the functor \(\rho\colon\Ia_1\to\funcat(\setcat,\setcat)\) where \(\rho(N)\coloneqq\binom{\_}{\le n}\).
\end{defn}

\begin{defn}[\(n^{\text{th}}\) Cartesian power functor]
Given \(n\in\N\) denote by \(\_^n\) the functor from \(\setcat\) to \(\setcat\) which takes a set \(A\) to the collection of \(n\)-tuples \(A^n\) over \(A\) and takes a function \(h\colon A\to B\) to the induced map from \(A^n\) to \(B^n\). We refer to \(\_^n\) as the \emph{\(n^{\text{th}}\) Cartesian power functor}.
\end{defn}

\begin{defn}[Cartesian signature]
Given an index set \(I\) and a function \(\tilde{\rho}\colon I\to\N\) the \emph{Cartesian signature} of \(\tilde{\rho}\) is the functor \(\rho\colon\Ia_I\to\funcat(\setcat,\setcat)\) given by \(\rho(N_i)\coloneqq\_^{\tilde{\rho}(i)}\).
\end{defn}

\begin{defn}[Powerset functor]
Denote by \(\pow\) the functor from \(\setcat\) to \(\setcat\) which takes a set \(A\) to the collection of subsets \(\pow(A)\) of \(A\) and takes a function \(h\colon A\to B\) to the induced map from \(\pow(A)\) to \(\pow(B)\). We refer to \(\pow\) as the \emph{powerset functor}.
\end{defn}

\begin{defn}[Hypergraph signature]
The \emph{hypergraph signature} is the functor \(\rho\colon\Ia_1\to\funcat(\setcat,\setcat)\) given by \(\rho(N)\coloneqq\pow\).
\end{defn}

\begin{defn}[Contravariant powerset functor]
Denote by \(\pow^\op\) the functor from \(\setcat\) to \(\setcat^\op\) which takes a set \(A\) to the collection of subsets \(\pow(A)\) of \(A\) and takes a function \(h\colon A\to B\) to the induced map from \(\pow(B)\) to \(\pow(A)\). We refer to \(\pow^\op\) as the \emph{contravariant powerset functor}.
\end{defn}

\begin{defn}[Pseudospace signature]
The \emph{pseudospace signature} is the functor \(\rho\colon\Ia_1\to\funcat(\setcat,\setcat^\op)\) given by \(\rho(N)\coloneqq\pow^\op\).
\end{defn}

Our central objects of study are manufactured from signatures.

\begin{defn}[Structure]
Given a \((\Ca,\Da)\)-signature \(\rho\) on an index category \(\Ia\) and \(A\in\ob(\Ca)\) we refer to a subobject \(\mathbf{A}\) of \(\rho_A\) in the category \(\funcat(\Ia,\Da)\) as a \emph{\((\Ca,\Da)\)-structure} of signature \(\rho\) on \(A\) (or as a \emph{\(\rho\)-structure} when we want to emphasize the signature). When \(\Ca=\Da\) we refer to a \((\Ca,\Da)\)-structure as a \emph{\(\Ca\)-structure}.
\end{defn}

We will often indicate a structure by giving a member of the corresponding equivalence class of monomorphisms into \(\rho_A\). That is, we will introduce a structure \(\mathbf{A}\) of signature \(\rho\) by saying something like ``consider a \((\Ca,\Da)\)-structure \(\mathbf{A}\) of signature \(\rho\) containing \(F\)'' where \(F\) is a monomorphism in \(\funcl(\Ia,\Da)\) with codomain \(\rho_A\) and \(\mathbf{A}\) is understood to be the equivalence class of monomorphisms which is the corresponding subobject of \(\rho_A\). If we want to be even more succinct we will write \(\mathbf{A}\coloneqq(A,F)\) where \(F\) is a monomorphism with codomain \(\rho_A\).

\subsection{Parts of a structure}
We name the various basic parts of a structure.

\begin{defn}[Universe]
Given a structure \(\mathbf{A}\) on an object \(A\) we refer to \(A\) as the \emph{universe} of \(\mathbf{A}\).
\end{defn}

\begin{defn}[Relation, arity of a relation]
Given a structure \(\mathbf{A}\) on an object \(A\) of signature \(\rho\colon\Ia\to\funcat(\Ca,\Da)\) and \(N\in\ob(\Ia)\) we refer to the class of morphisms \(\mathbf{A}_N\coloneqq\set[F_N]{F\in\mathbf{A}}\) in \(\Da\) as the \emph{relation} of \(\mathbf{A}\) at \(N\). We say that \(\mathbf{A}_N\) has \emph{arity} \(\rho(N)\) or that \(\mathbf{A}_N\) is \emph{\(\rho(N)\)-ary}.
\end{defn}

There is a corresponding idea for morphisms of \(\Ia\). Given a category \(\Da\) we denote by \(\morc(\Da)\) the morphism category whose objects are the morphisms of \(\Da\) and whose morphisms are natural transformations between the corresponding diagrams in \(\Da\). Given a natural transformation \(\eta\colon X\to Y\) of functors from \(\Ia\) to \(\Da\) and a morphism \(\nu\colon N_1\to N_2\) in \(\Ia\) we obtain a morphism from \(X(\nu)\) to \(Y(\nu)\) in \(\morc(\Da)\), which we refer to as the \emph{component} \(\eta_\nu\) of \(\eta\) at \(\nu\) in analogy with the usual components of a natural transformation.

\begin{defn}[Relator, arity of a relator]
Given a structure \(\mathbf{A}\) on an object \(A\) of signature \(\rho\colon\Ia\to\funcat(\Ca,\Da)\) and \(\nu\in\mor(\Ia)\) we refer to the class of morphisms \(\mathbf{A}_\nu\coloneqq\set[F_\nu]{F\in\mathbf{A}}\) in \(\morc(\Da)\) as the \emph{relator} of \(\mathbf{A}\) at \(\nu\). We say that \(\mathbf{A}_\nu\) has \emph{arity} \(\rho(\nu)\) or that \(\mathbf{A}_\nu\) is \emph{\(\rho(\nu)\)-ary}.
\end{defn}

In many contexts it will happen that \(\mathbf{A}_N\) is actually a subobject of \(\rho_A(N)\). Traditionally relations on a set are defined without reference to a particular structure. One possible generalization of this is to take a relation on \(A\) of arity \(\rho(N)\) to be a subobject of \(\rho_A(N)\), but it is not clear that \(\mathbf{A}_N\) is always a subobject of \(\rho_A(N)\) for structures as we have defined them. Similar comments hold for an extrinsic definition of relators.

\begin{defn}[Source]
Given a \((\Ca,\Da)\)-structure \(\mathbf{A}\) we refer to \(\Ca\) as the \emph{source} of \(\mathbf{A}\) and say that \(\mathbf{A}\) is a \emph{\(\Ca\)-sourced} structure.
\end{defn}

\begin{defn}[Target]
Given a \((\Ca,\Da)\)-structure \(\mathbf{A}\) we refer to \(\Da\) as the \emph{target} of \(\mathbf{A}\) and say that \(\mathbf{A}\) is a \emph{\(\Da\)-targeted} structure.
\end{defn}

\subsection{Categories of structures}
We consider categories whose objects are structures with a common signature.

\begin{defn}[Similarity class]
Given a signature \(\rho\) we refer to the class of all structures of signature \(\rho\) as the \emph{\(\rho\) similarity class}, which we denote by \(\struct^\rho\).
\end{defn}

\begin{defn}[Similar structures]
We say that two structures \(\mathbf{A}\) and \(\mathbf{B}\) of the same signature \(\rho\) are \emph{similar structures} or that \(\mathbf{A}\) and \(\mathbf{B}\) are \emph{of the same similarity type}.
\end{defn}

A homomorphism from a structure \(\mathbf{A}\) to a structure \(\mathbf{B}\) of the same similarity type \(\rho\) should be a morphism \(h\) from the universe \(A\) of \(\mathbf{A}\) to the universe \(B\) of \(\mathbf{B}\) which ``respects the structure of \(\mathbf{A}\) and \(\mathbf{B}\)''. In order to formalize this we make use of the extractor of \(\rho\).

\begin{defn}[Image of a structure]
Given a \((\Ca,\Da)\)-signature \(\rho\), objects \(A,B\in\ob(\Ca)\), a morphism \(h\colon A\to B\), and a structure \(\mathbf{A}\) of signature \(\rho\) on \(A\) containing \(F\) we refer to \(h(\mathbf{A})\coloneqq\im(\rho_h\circ F)\) as the \emph{image} of \(\mathbf{A}\) under \(h\).
\end{defn}

Were we to use presignatures rather than signatures to define structures the image of \(\mathbf{A}\) under \(h\colon A\to B\) might not exist, in which case our lives would be much harder.

\begin{defn}[Morphism of structures]
Let \(\mathbf{A}\) be a structure on an object \(A\) of signature \(\rho\) and let \(\mathbf{B}\) be a structure on an object \(B\) of signature \(\rho\). We say that a morphism \(h\colon A\to B\) is a \emph{morphism} from \(\mathbf{A}\) to \(\mathbf{B}\) when \(h(\mathbf{A})\le\mathbf{B}\) as subobjects of \(\rho_B\). We write \(h\colon\mathbf{A}\to\mathbf{B}\) to indicate that \(h\) is a morphism from \(\mathbf{A}\) to \(\mathbf{B}\).
\end{defn}

\begin{defn}[Category of structures of signature \(\rho\)]
We denote by \(\structc^\rho\) the \emph{category of structures of signature \(\rho\)} (or the \emph{category of \(\rho\)-structures}) whose objects are the structures of similarity type \(\rho\colon\Ia\to\funcat(\Ca,\Da)\), whose morphisms are morphisms of structures, whose identity morphisms are those induced by the identity morphisms of \(\Ca\), and whose composition of morphisms is given by composition of underlying morphisms in \(\Ca\).
\end{defn}

In order to establish that \(\structc^\rho\) is indeed a category we need the following lemmas.

\begin{lem}[Composition is isotone]
\label{lem:composition_is_isotone}
Suppose that \(\Ca\) is a category with \(A_1\), \(A_2\), \(A_3\), and \(A_4\) objects of \(\Ca\), \(h_1\colon A_1\hookrightarrow A_3\), \(h_2\colon A_2\hookrightarrow A_3\), and \(h_3\colon A_3\to A_4\) such that \(h_1\le h_2\) in which \(\im(h_3\circ h_1)\) and \(\im(h_3\circ h_2)\) exist. We have that \(\im(h_3\circ h_1)\le\im(h_3\circ h_2)\).
\end{lem}

\begin{proof}
Let \((V_1,\theta_1,\psi_1)\) be an image triple for \(h_3\circ h_1\) and let \((V_2,\theta_2,\psi_2)\) be an image triple for \(h_3\circ h_2\). Since \(h_1\le h_2\) there exists a morphism \(h_4\colon A_1\to A_2\) such that \(h_1=h_2\circ h_4\). It follows that \[h_3\circ h_1=h_3\circ h_2\circ h_4=\psi_2\circ\theta_2\circ h_4\] so \((V_2,\theta_2\circ h_4,\psi_2)\) is an image candidate for \(h_3\circ h_1\). Since \((V_1,\theta_1,\psi_1)\) is an image triple for \(h_3\circ h_1\) there exists a morphism \(s\colon V_1\to V_2\) such that \(\psi_2\circ s=\psi_1\). This implies that \(\psi_1\le\psi_2\) and hence \(\im(h_3\circ h_1)\le\im(h_3\circ h_2)\).
\end{proof}

The situation described in the preceding proof is depicted in \autoref{fig:composition_is_isotone}.

\begin{figure}
    \begin{center}
        \begin{tikzcd}
            A_2 \ar[rrr, "\theta_2"] \ar[rd, hook, "h_2"] & & & V_2 \ar[ld, hook, swap, "\psi_2"] \\
            & A_3 \ar[r, "h_3"] & A_4 & \\
            A_1 \ar[uu, "h_4"] \ar[ru, hook, swap, "h_1"] \ar[rrr, "\theta_1"] & & & V_1 \ar[lu, hook, "\psi_1"] \ar[uu, swap, "s"]
        \end{tikzcd}
    \end{center}
    \caption{Composition is isotone}
    \label{fig:composition_is_isotone}
\end{figure}
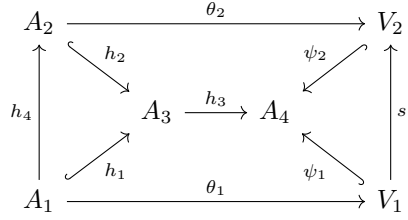

\begin{lem}[Morphism composition]
\label{lem:morphism_composition}
Suppose that \(\Ca\) is a category with \(A_i\in\ob(\Ca)\) for \(i\in\set{1,2,3,4,5,6}\), \(h_1\colon A_1\to A_2\), \(h_2\colon A_2\to A_3\), \(h_3\colon A_4\hookrightarrow A_1\), \(h_4\colon A_5\hookrightarrow A_2\), and \(h_5\colon A_6\hookrightarrow A_3\) such that \(\im(h_1\circ h_3)\), \(\im(h_2\circ h_4)\), and \(\im(h_2\circ h_1\circ h_3)\) exist. Suppose also that \(\psi_1\in\im(h_1\circ h_3)\), \(\im(h_2\circ\psi_1)\) exists, \(\im(h_1\circ h_3)\le\im(h_4)\), and \(\im(h_2\circ h_4)\le\im(h_5)\). We have that \(\im(h_2\circ h_1\circ h_3)\le\im(h_5)\).
\end{lem}

\begin{proof}
By the assumption that \(\im(h_1\circ h_3)\le\im(h_4)\) we have that \(\psi_1\le h_4\). Let \((V_1,\theta_1,\psi_1)\) be an image triple for \(h_1\circ h_3\), which must exist by our assumption that \(\psi_1\in\im(h_1\circ h_3)\). By \autoref{lem:composition_is_isotone} we have that \[\im(h_2\circ\psi_1)\le\im(h_2\circ h_4)\le\im(h_5).\] Since \(h_2\circ h_1\circ h_3=h_2\circ\psi_1\circ\theta_1\) it suffices to show that \(\im(h_2\circ\psi_1\circ\theta_1)\le\im(h_2\circ\psi_1)\).

Let \((V_2,\theta_2,\psi_2)\) be an image triple for \(h_2\circ\psi_1\circ\theta_1\) and let \((V_3,\theta_3,\psi_3)\) be an image triple for \(h_2\circ\psi_1\). Since \(h_2\circ\psi_1=\psi_3\circ\theta_3\) we have that \(h_2\circ\psi_1\circ\theta_1=\psi_3\circ\theta_3\circ\psi_1\) and hence \((V_3,\theta_3\circ\theta_1,\psi_3)\) is an image candidate for \(h_2\circ\psi_1\circ\theta_1\). By the universal property of the image triple of \(h_2\circ\psi_1\circ\theta_1\) we find that there exists a morphism \(s\colon V_2\to V_3\) such that \(\psi_3\circ s=\psi_2\). This implies that \(\psi_2\le\psi_3\) and hence \[\im(h_2\circ\psi_1\circ\theta_1)=\im(\psi_2)\le\im(\psi_3)=\im(h_2\circ\psi_1),\] as desired.
\end{proof}

The situation described in the preceding proof is depicted in \autoref{fig:morphism_composition}.

\begin{figure}
    \begin{center}
        \begin{tikzcd}
            & A_1 \ar[r, "h_1"] & A_2\ar[r, "h_2"] & A_3 \\
            A_4 \ar[ru, hook, "h_3"] \ar[r, "\theta_1"] & V_1 \ar[ru, hook, "\psi_1"] \ar[r] & A_5 \ar[u, hook, swap, "h_4"] & A_6 \ar[u, hook, "h_5"] \\
            & & & V_3 \ar[ld, hook, "\psi_3"] \\
            A_1 \ar[r, "\theta_1"] \ar[drrr, swap, "\theta_2"] & V_1 \ar[urr, "\theta_3"] \ar[r, swap, "h_2\circ\psi_1"] & A_3 & \\
            & & & V_2 \ar[lu, hook, swap, "\psi_2"] \ar[uu, "s"]
        \end{tikzcd}
    \end{center}
    \caption{Morphism composition}
    \label{fig:morphism_composition}
\end{figure}

We can now prove that structures form categories.

\begin{prop}
We have that \(\structc^\rho\) is a category for any signature \(\rho\).
\end{prop}

\begin{proof}
We show that morphisms compose. Let \(\mathbf{A}_i\coloneqq(A_i,F_i)\) with \(F_i\colon U_i\hookrightarrow\rho_{A_i}\) for \(i\in\set{1,2,3}\). Suppose that \(h_1\colon\mathbf{A}_1\to\mathbf{A}_2\) and \(h_2\colon\mathbf{A}_2\to\mathbf{A}_3\) are morphisms. We must establish that \(h_2\circ h_1\) is a morphism from \(\mathbf{A}_1\) to \(\mathbf{A}_3\), so we need that \(\im(h_2\circ h_2\circ F_1)\le\mathbf{A}_3\). Unraveling definitions we find that this is precisely the situation in \autoref{lem:morphism_composition}, so we have that \(h_2\circ h_1\colon\mathbf{A}_1\to\mathbf{A}_3\).

That composition of morphisms in \(\structc^\rho\) is associative follows directly from the associativity of composition in \(\Ca\). Similarly, identity morphisms in \(\structc^\rho\) satisfy the requisite identity because identity morphisms in \(\Ca\) do so.
\end{proof}

\begin{defn}[Kinship class]
Given a signature \(\rho\colon\Ia\to\funcat(\Ca,\Da)\) and an object \(A\) of \(\Ca\) we refer to the class of all structures of signature \(\rho\) with universe \(A\) as the \emph{\((\rho,A)\) kinship class}, which we denote by \(\struct^\rho_A\).
\end{defn}

\begin{defn}[Kindred structures]
We say that two structures \(\mathbf{A}\) and \(\mathbf{B}\) of the same similarity type with the same universe are \emph{kindred structures} or that \(\mathbf{A}\) and \(\mathbf{B}\) are \emph{of the same kinship type}.
\end{defn}

\subsection{Example: Pairs}
Take \(\rho\colon\Ia_1\to\funcat(\Ca,\Ca)\) to be the identity signature on a category \(\Ca\). Recall that by definition of the identity signature \(\Ca\) must have all images.

\begin{defn}[Pair in \(\Ca\)]
Given a category \(\Ca\) a \emph{pair} in \(\Ca\) (or a \emph{\(\Ca\)-pair}) is an ordered pair \((A,\im(F))\) where \(\im(F)\) is a subobject of \(A\) in \(\Ca\).
\end{defn}

\begin{defn}[Pair class in \(\Ca\)]
We refer to the class of all pairs in \(\Ca\) as the \emph{pair class in \(\Ca\)} (or the \emph{\(\Ca\)-pair class}), which we denote by \(\paircl(\Ca)\).
\end{defn}

We will usually write \((A,F)\) rather than \((A,\im(F))\) and remember that \((A,F_1)\) and \((A,F_2)\) are the same pair in \(\Ca\) when \(\im(F_1)=\im(F_2)\). Note that \((A_1,F_1)\neq(A_2,F_2)\) as pairs when \(A_1\neq A_2\), even if the domains of \(F_1\) and \(F_2\) are isomorphic.

\begin{defn}[Morphism of pairs in \(\Ca\)]
Given a category \(\Ca\), pairs \(\mathbf{A}_1\coloneqq(A_1,F_1)\) and \(\mathbf{A}_2\coloneqq(A_2,F_2)\), and a morphism \(h\colon A_1\to A_2\) we say that \(h\) is a \emph{morphism} from \(\mathbf{A}_1\) to \(\mathbf{A}_2\) and write \(h\colon\mathbf{A}_1\to\mathbf{A}_2\) when \(\im(h\circ F_1)\le\im(F_2)\).
\end{defn}

This is to say that a morphism of pairs is a morphism of the ambient objects \(A_1\) and \(A_2\) which takes \(\im(F_1)\) to \(\im(F_2)\).

\begin{defn}[Category of pairs in \(\Ca\)]
Given a category \(\Ca\) with all images the \emph{category of pairs in \(\Ca\)} (or the \emph{category of \(\Ca\)-pairs}) is the category \(\paircat(\Ca)\) whose objects are \(\Ca\)-pairs, whose morphisms are morphisms of pairs in \(\Ca\), for which the identity of \((A,F)\) is \(\id_A\), and whose composition is given by composition of morphisms in \(\Ca\).
\end{defn}

We need that \(\Ca\) has all images to show that \(\paircat(\Ca)\) is a category. If \(\Ca\) doesn't have all images then morphisms of pairs may not be composable even if their underlying morphisms in \(\Ca\) are composable.

\begin{prop}
Given a category \(\Ca\) with all images we have that \(\paircat(\Ca)\) is a category.
\end{prop}

\begin{proof}
We show that morphisms compose. Suppose that \(\mathbf{A}_i\coloneqq(A_i,F_i)\in\paircl(\Ca)\) for each \(i\in\set{1,2,3}\) and that \(h_1\colon\mathbf{A}_1\to\mathbf{A}_2\) and \(h_2\colon\mathbf{A}_2\to\mathbf{A}_3\) are morphisms of pairs. We show that \(h_2\circ h_1\colon A_1\to A_3\) is a morphism from \(\mathbf{A}_1\) to \(\mathbf{A}_3\). Since \(\Ca\) has all images this is precisely the situation in \autoref{lem:morphism_composition} so morphisms in \(\paircat(\Ca)\) compose.

Again the associativity of composition and the identity property for \(\paircat(\Ca)\) follow immediately from those for \(\Ca\).
\end{proof}

It is not surprising that the proof that \(\paircat(\Ca)\) is essentially identical to the proof that \(\structc^\rho\) is a category since by our characterization of the category \(\funcat(\Ia_1,\Ca)\) we find that \(\structc^\rho\cong\paircat(\Ca)\).

From this isomorphism we see that given a structure \(\mathbf{A}\coloneqq(A,F)\in\struct^\rho\) we have that the relation \(\mathbf{A}_N\) is the subobject \(\im(F)\) of \(A\) in \(\Ca\) and that \(\mathbf{A}\) has no nontrivial basic relators.

\subsection{Example: Hypergraphs}
We examine the category of structures obtained from the \(n\)-hypergraph signature \(\rho\colon\Ia_1\to\funcat(\setcat,\setcat)\).

\begin{defn}[\(n\)-hypergraph]
Given a set \(A\) we refer to \(\mathbf{A}\coloneqq(A,F)\) where \(F\subset\binom{A}{\le n}\) as an \emph{\(n\)-hypergraph} on \(A\).
\end{defn}

We denote by \(\hypercl_n\) the class of \(n\)-hypergraphs.

\begin{defn}[Morphism of \(n\)-hypergraphs]
Given \(n\)-hypergraphs \(\mathbf{A}_1\coloneqq(A_1,F_1)\) and \(\mathbf{A}_2\coloneqq(A_2,F_2)\) we refer to a function \(h\colon A_1\to A_2\) as a \emph{morphism} from \(\mathbf{A}_1\) to \(\mathbf{A}_2\) and write \(h\colon\mathbf{A}_1\to\mathbf{A}_2\) when \(h(F_1)\subset F_2\) where \[h(F_1)\coloneqq\set[\set[h(a)]{a\in E}]{E\in F_1}.\]
\end{defn}

\begin{defn}[Category of \(n\)-hypergraphs]
We denote by \(\hypercat_n\) the \emph{category of \(n\)-hypergraphs} whose objects form the class \(\hypercl_n\), whose morphisms are morphisms of \(n\)-hypergraphs, for which the identity of \((A,F)\) is \(\id_A\), and whose composition is given by composition of functions.
\end{defn}

It is evident that \(\structc^\rho\cong\hypercat_n\).

\subsection{Further examples}
When \(\rho\) is the hypergraph signature given by \(\rho(N)\coloneqq\pow{}\), we have that \(\structc^\rho\) is the category of hypergraphs, which has the category of simplicial sets as a subcategory. When \(\rho\) is instead the pseudospace signature given by \(\rho(N)\coloneqq\pow{}^\op\), we have that \(\structc^\rho\) has the category of topological spaces as a subcategory. We leave it to the reader to unravel what a pseudospace is supposed to be (concretely) from this abstract definition.

\section{The Yoneda embedding}
\label{sec:Yoneda}
Although our definition of structure appears to be more general than the structures usually considered in model theory, whose basic relations are subsets of Cartesian powers of the universe, we show that each category of \(\setcat\)-sourced structures embeds into a category of \(\setcat\)-structures whose basic relations are all subsets of Cartesian powers of the universe, at the expense that our new index category may be large where our original index category was small.

The driving device here is the Yoneda embedding. Given a locally small category \(\Ca\) let \(\yo_{\_}\colon\Ca\to\funcat(\Ca^\op,\setcat)\) denote the contravariant \(\hom\)-functor. Recall the following embedding of categories due to Yoneda.

\begin{lem}[Yoneda Lemma]
Let \(\Ca\) be a locally small category. The functor \(\yo_{\_}\) is full and faithful.
\end{lem}

We actually need more general source categories than \(\setcat\).

\begin{defn}[Exponential category]
We say that a full subcategory \(\Ca\) of \(\setcat\) is \emph{exponential} when \(\Ca\) is closed under taking subsets and forming exponential objects.
\end{defn}

One example of an exponential category is the full subcategory of \(\setcat\) whose objects are the empty set and every singleton set. The largest possible example of an exponential category is \(\setcat\) itself. We will be most interested in the exponential category \(\finsetcat\) whose objects form the class \(\finsetcl\) of finite sets.

Given an exponential category \(\Ca\) and a signature \(\rho\colon\Ia\to\funcat(\Ca,\Da)\) we apply the Yoneda Lemma to \(\structc^\rho\) to obtain an embedding \[\yo_{\_}\colon\structc^\rho\hookrightarrow\funcat((\structc^\rho)^\op,\setcat).\] Each structure \(\mathbf{A}\in\struct^\rho_A\) determines a functor \[\yo_{\mathbf{A}}\colon(\structc^\rho)^\op\to\setcat.\] If \(\mathbf{A}\in\struct^\rho_A\) and \(\mathbf{C}\in\struct^\rho_C\) with \(A,C\in\ob(\Ca)\) then \[\yo_{\mathbf{A}}(\mathbf{C})=\hom(\mathbf{C},\mathbf{A})\subset A^C\] so \(\yo_{\mathbf{A}}\) can be restricted on its codomain to a functor from \((\structc^\rho)^\op\) to \(\Ca\).

\begin{defn}[Exponential Yoneda functor]
Given an exponential category \(\Ca\) and a signature \(\rho\colon\Ia\to\funcat(\Ca,\Da)\) the \emph{exponential Yoneda functor} \[\yo^\Ca_{\_}\colon\structc^\rho\to\funcat((\structc^\rho)^\op,\Ca)\] of \(\rho\) over \(\Ca\) is the functor obtained by restricting the codomain of \(\yo_{\mathbf{A}}\) to \(\Ca\) for each \(\mathbf{A}\in\struct^\rho\).
\end{defn}

Restricting the codomain of a functor preserves embeddings so the exponential Yoneda functor is also an embedding, which we are justified in calling the \emph{exponential Yoneda embedding}.

\begin{defn}[Yoneda signature]
Given an exponential category \(\Ca\) and a signature \(\rho\colon\Ia\to\funcat(\Ca,\Da)\) the \emph{Yoneda signature} \(\rho^\Ca\) of \(\rho\) over \(\Ca\) is the functor \(\rho^\Ca\colon(\structc^\rho)^\op\to\funcat(\Ca,\Ca)\) defined as follows. For \(\mathbf{C}\in\struct^\rho_C\) the functor \(\rho^\Ca(\mathbf{C})\colon\Ca\to\Ca\) is given by \(\rho^\Ca(\mathbf{C})(A)\coloneqq A^C\) for each set \(A\) and \[\rho^\Ca(\mathbf{C})(h)\coloneqq h\circ\_\colon A_1^C\to A_2^C\] for each function \(h\colon A_1\to A_2\). Given a morphism \(f\colon\mathbf{C}_2\to\mathbf{C}_1\) in \(\structc^\rho\) the natural transformation \(\rho^\Ca(f)\colon\rho^\Ca(\mathbf{C}_1)\to\rho^\Ca(\mathbf{C}_2)\) is given by \[\rho^\Ca(f)_A\coloneqq\_\circ f\colon\rho^\Ca(\mathbf{C}_1)(A)\to\rho^\Ca(\mathbf{C}_2)(A).\]
\end{defn}

In order to show that \(\rho^{\Ca}\) is a signature we need the following lemma about exponential categories.

\begin{lem}
\label{lem:exp_cat_has_factorization}
Given an exponential category \(\Ca\) and a function \(h\colon A\to B\) in \(\Ca\) we have that \(h\) has an image triple \((V,\theta,\psi)\) in \(\Ca\) where \(\theta\) is surjective.
\end{lem}

\begin{proof}
Since \(\Ca\) is closed under taking subobjects we can form the subset \(V\subset B\) given by \(V\coloneqq\set[b\in B]{(\exists a\in A)(h(a)=b)}\). Define \(\theta\colon A\to V\) by \(\theta(a)\coloneqq h(a)\) for each \(a\in A\) and let \(\psi\colon V\hookrightarrow B\) be the inclusion of \(V\) as a subset of \(B\). Since \(\Ca\) is a full subcategory of \(\setcat\) we have that \((V,\theta,\psi)\) remains an image triple for \(h\) in \(\Ca\). Observe that \(\theta\) is surjective.
\end{proof}

\begin{prop}
The Yoneda signature \(\rho^\Ca\) is a signature.
\end{prop}

\begin{proof}
We show that \(\rho^\Ca\) is a functor. For \(\mathbf{C}\in\struct^\rho_C\) we show that \(\rho^\Ca(\mathbf{C})\colon\Ca\to\Ca\) is a functor and hence an object of \(\funcat(\Ca,\Ca)\).

Given \(A\in\ob(\Ca)\) we have that \(\rho^\Ca(\mathbf{C})(A)=A^C\), which is an object of \(\Ca\) since \(A,C\in\ob(\Ca)\) and \(\Ca\) is an exponential category. Thus, \(\rho^\Ca(\mathbf{C})\) takes objects to objects. Given a function \(h\colon A_1\to A_2\) in \(\Ca\) we have that \(\rho^\Ca(\mathbf{C})(h)\colon A_1^C\to A_2^C\) is a function. Thus, \(\rho^\Ca(\mathbf{C})\) takes morphisms to morphisms. Given an identity map \(\id_A\colon A\to A\) in \(\Ca\) we have that \[\rho^\Ca(\mathbf{C})(\id_A)=\id_A\circ\_=\id_{A^C}\colon A^C\to A^C.\] Thus, \(\rho^\Ca(\mathbf{C})\) takes identities to identities. Given functions \(h_1\colon A_1\to A_2\) and \(h_2\colon A_2\to A_3\) in \(\Ca\) we have that
	\begin{align*}
		\rho^\Ca(\mathbf{C})(h_2\circ h_1) &= (h_2\circ h_1)\circ\_ \\
		&= h_2\circ(h_1\circ\_) \\
		&= (h_2\circ\_)\circ(h_1\circ\_) \\
		&= \rho^\Ca(\mathbf{C})(h_2)\circ\rho^\Ca(\mathbf{C})(h_1)
	\end{align*}
so \(\rho^\Ca(\mathbf{C})\) respects composition. We find that \(\rho^\Ca(\mathbf{C})\) is a functor and hence \(\rho^\Ca\) takes objects to objects.

For \(f\colon\mathbf{C}_2\to\mathbf{C}_1\) in \(\structc^\rho\) we show that \(\rho^\Ca(f)\colon\rho^\Ca(\mathbf{C}_1)\to\rho^\Ca(\mathbf{C}_2)\) is a natural transformation and hence a morphism of \(\funcat(\Ca,\Ca)\). Given a function \(h\colon A_1\to A_2\) in \(\Ca\) it is immediate that \[\rho^\Ca(f)_{A_2}\circ\rho^\Ca(\mathbf{C}_1)(h)=h\circ\_\circ f=\rho^\Ca(\mathbf{C}_2)(h)\circ\rho^\Ca(f)_{A_1}\] so \(\rho^\Ca(f)\) is a natural transformation and hence \(\rho^\Ca\) takes morphisms to morphisms.

Given an identity morphism \(\id_{\mathbf{C}}\colon\mathbf{C}\to\mathbf{C}\) and any \(A\in\ob(\Ca)\) we have that \(\rho^\Ca(\id_{\mathbf{C}})_A=\_\circ\id_C=\id_{A^C}\) so \(\rho^\Ca(\id_{\mathbf{C}})=\id_{\rho^\Ca(\mathbf{C})}\) is the identity natural transformation of \(\rho^\Ca(\mathbf{C})\) and \(\rho^\Ca\) takes identities to identities.

Given morphisms \(f_2\colon\mathbf{C}_3\to\mathbf{C}_2\) and \(f_1\colon\mathbf{C}_2\to\mathbf{C}_1\) in \(\structc^\rho\) and \(A\in\ob(\Ca)\) we have that \[\rho^\Ca(f_2)_A\circ\rho^\Ca(f_1)_A=\_\circ f_1\circ f_2=\rho^\Ca(f_1\circ f_2)_A\] so \(\rho^\Ca\) respects composition and is thus a functor from \((\structc^\rho)^\op\) to \(\funcat(\Ca,\Ca)\).

It remains to show that given a monomorphism \(F\colon U\hookrightarrow\rho^\Ca_{A_1}\) in
    \[
        \funcat((\structc^\rho)^\op,\Ca)
    \]
and any function \(h\colon A_1\to A_2\) in \(\Ca\) we have that \(\im(\rho^\Ca_h\circ F)\) exists in
    \[
        \funcat((\structc^\rho)^\op,\Ca).
    \]

For each \(\mathbf{C}\in\ob(\Ca)\) let \((V_{\mathbf{C}},\theta_{\mathbf{C}},\psi_{\mathbf{C}})\) be an image triple for \((\rho^\Ca_h\circ F)_{\mathbf{C}}\) where \(\theta_{\mathbf{C}}\) is surjective as guaranteed by \autoref{lem:exp_cat_has_factorization}. For each morphism \(f\colon\mathbf{C}_2\to\mathbf{C}_1\) we have that
	\begin{align*}
		\im(\rho^\Ca_{A_2}(f)\circ\psi_{\mathbf{C}_1}) &= \im(\rho^\Ca_{A_2}(f)\circ\psi_{\mathbf{C}_1}\circ\theta_{\mathbf{C}_1}) \\
		&= \im(\psi_{\mathbf{C}_2}\circ\theta_{\mathbf{C}_2}\circ U(f)) \\
		&\le \im(\psi_{\mathbf{C}_2})
	\end{align*}
so the codomain restriction \[g_f\coloneqq(\rho^\Ca_{A_2}(f)\circ\psi_{\mathbf{C}_1})|_{V_{\mathbf{C}_2}}\colon V_{\mathbf{C}_1}\to V_{\mathbf{C}_2}\] of \(\rho^\Ca_{A_2}(f)\circ\psi_{\mathbf{C}_1}\) to \(V_{\mathbf{C}_2}\) exists in \(\Ca\).

We use this data to factor \(\rho^\Ca_h\circ F\). Define a functor \(V\colon(\structc^\rho)^\op\to\Ca\) by \(V(\mathbf{C})\coloneqq V_{\mathbf{C}}\) for each \(\mathbf{C}\in\struct^\rho\) and \(V(f)\coloneqq g_f\) for each morphism \(f\colon\mathbf{C}_2\to\mathbf{C}_1\) in \(\structc^\rho\). Define natural transformations \(\theta\colon U\to V\) and \(\psi\colon V\to\rho^\Ca_{A_2}\) whose components at \(\mathbf{C}\) are \(\theta_{\mathbf{C}}\) and \(\psi_{\mathbf{C}}\), respectively.

We show that \(V\colon(\structc^\rho)^\op\to\Ca\) is a functor. Given \(\mathbf{C}\in\struct^\rho\) we have that \(V(\mathbf{C})=V_{\mathbf{C}}\in\ob(\Ca)\) by definition of \(V\) so \(V\) takes objects to objects. Given a morphism \(f\colon\mathbf{C}_2\to\mathbf{C}_1\) in \(\structc^\rho\) we have that \(V(f)=g_f\) is a morphism in \(\Ca\) by definition so \(V\) takes morphisms to morphisms. Given an object \(\mathbf{C}\in\struct^\rho_C\) and its identity morphism \(\id_{\mathbf{C}}\colon\mathbf{C}\to\mathbf{C}\) we have that \(\rho^\Ca\) is a presignature so by \autoref{prop:extractor_is_functor} we have that \(\rho^\Ca_{A_2}\) is a functor and hence \(\rho^\Ca_{A_2}(\id_{\mathbf{C}})=\id_{\rho^\Ca_{A_2}(\mathbf{C})}\). It follows that \[V(\id_{\mathbf{C}})=g_{\id_{\mathbf{C}}}=(\rho^\Ca_{A_2}(\id_{\mathbf{C}})\circ\psi_{\mathbf{C}})|_{V_{\mathbf{C}}}=(\id_{\rho^\Ca_{A_2}(\mathbf{C})}\circ\psi_{\mathbf{C}})|_{V_{\mathbf{C}}}=\psi_{\mathbf{C}}|_{V_{\mathbf{C}}}=\id_{V(\mathbf{C})}\] so \(V\) takes identities to identities. Given morphisms \(f_2\colon\mathbf{C}_3\to\mathbf{C}_2\) and \(f_1\colon\mathbf{C}_2\to\mathbf{C}_1\) in \(\structc^\rho\) we have that \(\im(\rho^\Ca_{A_2}(f_1)\circ\psi_{\mathbf{C}_1})\le\im(\psi_{\mathbf{C}_2})\) and hence
	\begin{align*}
		V(f_1\circ f_2) &= g_{f_1\circ f_2} \\
		&= (\rho^\Ca_{A_2}(f_1\circ f_2)\circ\psi_{\mathbf{C}_1})|_{V_{\mathbf{C}_3}} \\
		&= (\rho^\Ca_{A_2}(f_2)\circ\rho^\Ca_{A_2}(f_1)\circ\psi_{\mathbf{C}_1})|_{V_{\mathbf{C}_3}} \\
		&= (\rho^\Ca_{A_2}(f_2)\circ\psi_{\mathbf{C}_2})|_{V_{\mathbf{C}_3}}\circ(\rho^\Ca_{A_2}(f_1)\circ\psi_{\mathbf{C}_1})|_{V_{\mathbf{C}_2}} \\
		&= V(f_2)\circ V(f_1).
	\end{align*}
Thus, \(V\) respects composition and is a functor from \((\structc^\rho)^\op\) to \(\Ca\).

It is evident that \(\theta\) and \(\psi\) are natural transformations.

Since each of the components of \(\psi\) are injective we have that \(\psi\) is monic. By definition we have that \(\rho^\Ca_h\circ F=\psi\circ\theta\). It follows that \(\im(\rho^\Ca\circ F)\) exists in \(\funcat((\structc^\rho)^\op,\Ca)\) and contains \(\psi\).
\end{proof}

We can use the exponential Yoneda embedding to obtain a functor from \(\structc^\rho\) to the category of \(\Ca\)-structures \(\structc^{\rho^\Ca}\).

\begin{defn}[Cartesian Yoneda functor]
Given an exponential category \(\Ca\) and a signature \(\rho\colon\Ia\to\funcat(\Ca,\Da)\) the \emph{Cartesian Yoneda functor} \[\cyo\colon\structc^\rho\to\structc^{\rho^\Ca}\] is defined as follows. Given \(\mathbf{A}\in\struct^\rho_A\) we define \(\cyo(\mathbf{A})\) to be the subobject of \(\rho^\Ca_A\) containing the monomorphism \(F_{\mathbf{A}}\colon\yo^\Ca_{\mathbf{A}}\hookrightarrow\rho^\Ca_A\) given by \((F_{\mathbf{A}})_{\mathbf{C}}\colon\yo^\Ca_{\mathbf{A}}(\mathbf{C})\to\rho^\Ca_A(\mathbf{C})\) where \((F_{\mathbf{A}})_{\mathbf{C}}(f)\coloneqq f\). Given \(h\colon\mathbf{A}_1\to\mathbf{A}_2\) we define \(\cyo(h)\coloneqq h\).
\end{defn}

Intuitively, the Cartesian Yoneda functor is the natural inclusion of \(\hom(\mathbf{C},\mathbf{A})\) into \(A^C\).

We will need the following lemma about images in general categories which gives a sufficient condition for the image of structure under a morphism to be contained in another structure.

\begin{lem}
\label{lem:subobject_containment}
Let \(\Ca\) be a category with \(X,Y,Z\in\ob(\Ca)\) and morphisms \(\alpha\colon X\to Z\), \(\beta\colon X\to Y\), and \(\gamma\colon Y\hookrightarrow Z\) such that \(\alpha=\gamma\circ\beta\). Let \(U\in\ob(\Ca)\) with \(\theta_X\colon X\to U\) and \(\theta_Z\colon U\hookrightarrow Z\) a factorization of \(\alpha\) witnessing that \(\theta_Z\) is the image of \(\alpha\). We have that \(\theta_Z\le\gamma\).
\end{lem}

\begin{proof}
Since \(\beta\colon X\to Y\) and \(\gamma\colon Y\hookrightarrow Z\) form a factorization of \(\alpha\) we have by definition of the image \(\theta_Z\) that there exists a morphism \(s\colon U\to Y\) such that \(\theta_Z=\gamma\circ s\). Thus, \(\theta_Z\le\gamma\).
\end{proof}

\begin{prop}
We have that \(\cyo\colon\structc^\rho\to\structc^{\rho^\Ca}\) is a functor.
\end{prop}

\begin{proof}
We show that \(\cyo\) takes objects to objects. Given \(\mathbf{A}\in\struct^\rho_A\) we must show that \(F_{\mathbf{A}}\colon\yo^\Ca_{\mathbf{A}}\to\rho^\Ca_A\) is a monomorphism. Let \(U\colon(\structc^\rho)^\op\to\Ca\) be a functor and let \(H_1,H_2\colon U\to\yo^\Ca_{\mathbf{A}}\) be natural transformations. We show that \(F_{\mathbf{A}}\circ H_1=F_{\mathbf{A}}\circ H_2\) implies that \(H_1=H_2\).

Fix an object \(\mathbf{C}\in\struct^\rho_C\). We have that \((\yo^\Ca_{\mathbf{A}})({\mathbf{C}})=\hom(\mathbf{C},\mathbf{A})\) and \((\rho^\Ca_A)({\mathbf{C}})=A^C\). By definition we find that \((F_{\mathbf{A}})_{\mathbf{C}}\colon\hom(\mathbf{C},\mathbf{A})\to A^C\) is the map taking a morphism \(h\colon\mathbf{C}\to\mathbf{A}\) to its underlying function \(h\colon C\to A\). We have that \(U(\mathbf{C})\) is a set and \((H_1)_{\mathbf{C}},(H_2)_{\mathbf{C}}\colon U(\mathbf{C})\to\hom(\mathbf{C},\mathbf{A})\) are functions. Since \((F_{\mathbf{A}})_{\mathbf{C}}\) is injective we have that \((H_1)_{\mathbf{C}}=(H_2)_{\mathbf{C}}\). As the components of \(H_1\) and \(H_2\) are the same we have that \(H_1=H_2\). Thus, \(F_{\mathbf{A}}\) is a monomorphism and \(\cyo\) does take objects of \(\structc^\rho\) to objects of \(\structc^{\rho^\Ca}\).

We show that \(\cyo\) takes morphisms to morphisms. Given \(\mathbf{A}_1\in\struct^\rho_{A_1}\), \(\mathbf{A}_2\in\struct^\rho_{A_2}\), and a morphism \(h\colon\mathbf{A}_1\to\mathbf{A}_2\) we must show that \[\cyo(h)(\cyo(\mathbf{A}_1))\le\cyo(\mathbf{A}_2)\] as subobjects of \(\rho^\Ca_{A_2}\).

Choose representative monomorphisms \(F_{\mathbf{A}_1}\) and \(F_{\mathbf{A}_2}\) of \(\cyo(\mathbf{A}_1)\) and \(\cyo(\mathbf{A}_2)\), respectively. By \autoref{lem:subobject_containment} it suffices to show that there exists a natural transformation \(\eta\colon\yo^\Ca_{\mathbf{A}_1}\to\yo^\Ca_{\mathbf{A}_2}\) such that \(\rho^\Ca_h\circ F_{\mathbf{A}_1}=F_{\mathbf{A}_2}\circ\eta\). We claim that we can take \(\eta=\yo^\Ca_h\). Unraveling definitions we find that for each \(\mathbf{C}\in\struct^\rho\) and each \(f\in\yo^\Ca_{\mathbf{A}_1}(\mathbf{C})\) we have \[(\rho^\Ca_h\circ F_{\mathbf{A}_1})_{\mathbf{C}}(f)=h\circ f=(F_{\mathbf{A}_2}\circ\yo^\Ca_h)_{\mathbf{C}}(f),\] as claimed.

We show that \(\cyo\) takes identities to identities. Given \(\mathbf{A}\in\struct^\rho_A\) and \(\id_{\mathbf{A}}\colon\mathbf{A}\to\mathbf{A}\) we have that \(\cyo(\id_{\mathbf{A}})=\id_A\), which is the identity morphism of \(\cyo(\mathbf{A})\) in \(\structc^{\rho^\Ca}\).

We show that \(\cyo\) respects the composition of morphisms. Since \(\cyo\) maps a morphism \(h\colon\mathbf{A}_1\to\mathbf{A}_2\) to its underlying function \(h\colon A_1\to A_2\) and in both \(\structc^\rho\) and \(\structc^{\rho^\Ca}\) composition of morphisms is given by composition of the underlying functions we have that \(\cyo\) respects composition.
\end{proof}

The Cartesian Yoneda functor is an embedding of categories.

\begin{prop}
\label{prop:cartesian_yo_full_faith}
The functor \(\cyo\colon\structc^\rho\to\structc^{\rho^\Ca}\) is full and faithful.
\end{prop}

\begin{proof}
We show that \(\cyo\) is full. Suppose that \(h\in\hom(\cyo(\mathbf{A}_1),\cyo(\mathbf{A}_2))\). By definition we have that \(\rho^\Ca(h)(\cyo(\mathbf{A}_1))\le\cyo(\mathbf{A}_2)\) so there exists some \(\eta\colon\yo^\Ca_{\mathbf{A}_1}\to\yo^\Ca_{\mathbf{A}_2}\) such that \(\rho^\Ca_h\circ F_{\mathbf{A}_1}=F_{\mathbf{A}_2}\circ\eta\) for representative monomorphisms \(F_{\mathbf{A}_1}\) and \(F_{\mathbf{A}_2}\) of \(\mathbf{A}_1\) and \(\mathbf{A}_2\), respectively. Since \(\id_{A_1}\in\hom(\mathbf{A}_1,\mathbf{A}_1)\) this implies that \[h = h\circ\id_{A_1} = (\rho^\Ca_h\circ F_{\mathbf{A}_1})_{\mathbf{A}_1}(\id_{A_1})	= (F_{\mathbf{A}_2}\circ\eta)_{\mathbf{A}_1}(\id_{A_1})\]
from which it follows that \(h\) is in the image of \[(F_{\mathbf{A}_2})_{\mathbf{A}_1}\colon\hom(\mathbf{A}_1,\mathbf{A}_2)\to A_2^{A_1}.\] Thus, \(h\in\hom(\mathbf{A}_1,\mathbf{A}_2)\).

Note that since \(\cyo(h)\coloneqq h\) we have that \(\cyo\) is faithful.
\end{proof}

We are thus justified in referring to \(\cyo\) as the \emph{Cartesian Yoneda embedding}. As a special case, we have proven our remark from the beginning of this section. Given a signature \(\rho\colon\Ia\to\funcat(\setcat,\Da)\) we have that \(\cyo\) is an embedding of \(\structc^\rho\) into \(\structc^{\rho^{\setcat}}\). Given objects \(\mathbf{A}\in\struct^\rho_A\) and \(\mathbf{C}\in\struct^\rho_A\) the relation of \(\cyo(\mathbf{A})\) at \(\mathbf{C}\) is a subset of \(A^C\), or a \(C\)-ary relation on \(A\) in the classical sense.

\printbibliography

\end{document}